\documentclass[usletter]{amsart} 

\usepackage{amsmath,amsthm,amssymb,amsfonts,mathrsfs,color,hyperref, mathtools, graphicx, enumitem, todonotes}
\usepackage[letterpaper, left= 3cm, right=3cm, top= 3cm, bottom = 3cm]{geometry}
\usepackage{apptools}
\AtAppendix{\counterwithin{theorem}{section}}

\theoremstyle{plain}
\begingroup
\newtheorem{theorem}{Theorem}[]
\newtheorem*{theorem*}{Theorem}
\newtheorem*{"theorem"}{``Theorem''}

\newtheorem{lemma}[theorem]{Lemma}
\endgroup

\theoremstyle{definition}
\begingroup

\endgroup

\theoremstyle{remark}
\begingroup
\newtheorem{remark}[theorem]{Remark}

\endgroup 

\numberwithin{equation}{section}
\newenvironment{pde}{\left\{\begin{array}{rll} } {\end{array}\right.}

\newcommand{\N}{\mathbb N}

\newcommand{\R}{\mathbb R} 
\newcommand{\E}{{\mathbb E}}

\newcommand{\dist}{{\rm dist}}
\newcommand{\diam}{{\rm diam}}

\renewcommand{\H}{{\mathcal H}}

\newcommand{\Ra} {\Rightarrow}
\newcommand{\wto}{\rightharpoonup}
\newcommand{\embeds}{\xhookrightarrow{\quad}}

\renewcommand{\d}{\mathrm{d}}

\newcommand{\dx}{\,\mathrm{d}x}

\newcommand{\dz}{\,\mathrm{d}z}
\newcommand{\ds}{\,\mathrm{d}s}
\newcommand{\dt}{\,\mathrm{d}t}
\newcommand{\dr}{\,\mathrm{d}r}

\newcommand{\eps}{\varepsilon}
\newcommand{\average}{{\mathchoice {\kern1ex\vcenter{\hrule height.4pt
width 6pt depth0pt} \kern-9.7pt} {\kern1ex\vcenter{\hrule
height.4pt width 4.3pt depth0pt} \kern-7pt} {} {} }}

\allowdisplaybreaks

 \makeatletter
\@namedef{subjclassname@2020}{2020 Mathematics Subject Classification}
\makeatother
 \newcommand{\showlabel}{\addtocounter{equation}{1}\tag{\theequation}}

\newcommand{\B}{\mathcal{B}}
\usepackage{afterpage}

\begin{document}

\title[Elliptic Regularity in Barron Spaces]{Elliptic Regularity Theory in Barron Spaces and Applications to the Deep Ritz Method}

\author{Stephan Wojtowytsch}
\address{Stephan Wojtowytsch\\
University of Pittsburgh\\
Department of Mathematics\\
Thackeray Hall\\
Pittsburgh, PA 15213
}
\email{s.woj@pitt.edu}

\date{\today}


\begin{abstract}
We prove that harmonic functions with Dirichlet boundary data in Barron space, a function class tailored to wide ReLU networks with a single hidden layer and suitably bounded weights, are generally neither Lipschitz continuous nor in the Sobolev class $H^2$. A fortiori, they are not in any function class in which the norm controls the Lipschitz constant, which rules out not only Barron space regularity, but also regularity in function classes for deeper ReLU networks with bounded coefficients. They can, however, be approximated to accuracy $\sim \varepsilon$ by Barron functions of low norm $\sim |\log\varepsilon|$ in various Lebesgue and Sobolev norms (with at most two derivatives). The positive result holds on very simple domains: Half-spaces in arbitrary dimension and rectangular domains in two dimensions. As an application of this regularity theory, we obtain a priori error estimates for Deep Ritz neural PDE solvers.
\end{abstract}

\maketitle

%\setcounter{tocdepth}{1}
%\tableofcontents

\section{Introduction}

In recent years, neural networks have gained popularity in scientific computing. Among others, these applications encompass computational tools where the neural network approximates the solution of a partial differential equation (PDE) such as the Deep Ritz method \cite{weinan2018deep, dondl2022uniform} and physics-informed neural networks \cite{raissi2019physics}, as well as those where it approximates the solution operator of a parametric PDE between function classes \cite{li2020fourier, kovachki2021neural, kovachki2021universal, wen2022u, helwig2023group}. Also in the second case, the solution in space (or space-time) may be represented by a neural network -- this is the case for instance for Deep Operator Networks \cite{lu2019deeponet, liu2024neural, weihs2025deep}.

In this work, we investigate the capacity of neural networks to represent and/or approximate solutions of PDEs if the problem data are given or well-approximated by neural networks in the arguably simplest instance: the Poisson equation with Dirichlet boundary values
\begin{equation}
\begin{pde} \label{eq poisson equation intro}
-\Delta u &= f&\text{in }\Omega\\ u &= g &\text{on }\partial\Omega.
\end{pde}
\end{equation}
Whether the solution of a problem cannot be represented exactly or approximated efficiently has crucial impact on the training process -- see e.g.\ \cite{dynamic_cod} and \cite[Corollary 5.6]{deep_barron}. While even exact representation may not be sufficient for successful training \cite{livni2014computational}, efficient approximability is generally considered necessary.
  
This article addresses the question in the setting of {\em Barron spaces} $\mathcal B$, function spaces designed to capture the functions which can be represented by neural networks with a single hidden layer (perceptra) and ReLU activation. These function spaces contain all functions represented by finite ReLU perceptra $\frac1m \sum_{i=1}^ma_i\,\sigma(w_i\cdot x+b_i)$ as well as their infinite width limits $\E_{(a,w,b)\sim\pi}[ a\,\sigma(w\cdot x+b)]$, assuming that the average magnitude of their weights $\E_{(a,w,b)}\big[|a|^2 + |w|^2 + b^2\big]$ remains bounded. Dynamically, this `norm bound' would correspond to a weight decay regularizer. In the infinite width limit, the empirical average (sum) is replaced with a true expectation (integral) over a probability distribution on the space of parameters. A more comprehensive introduction is given in the next section.

%Efficient approximability may have a positive impact on the solvability of the optimization problem, although it is far from a guarantee. 
Philosophically, regularity theory in Barron spaces addresses the question: If the problem data can be approximated well by shallow ReLU networks, can the solution be as well? This is particularly natural for Dirichlet boundary data, which are usually learned from observations by {\em the same} neural network that represents the solution to the PDE in the interior. We consider the case $f= \Delta U$ for $U\in \B$ to ensure that both the PDE and the boundary condition can be satisfied individually in Barron space. The question remains: If both can be satisfied {\em individually} in Barron space, is it possible to satisfy them {\em simultaneously}?

The answer is negative, but in a quantitatively mild way.
Our main regularity result can be summarized as follows.

\begin{theorem}[Regularity theory with Barron data] \label{theorem main}
\begin{enumerate}
\item Let $g, U\in\B$ and $\Omega\subseteq\R^d$ a bounded open set. Then the problem
\begin{equation}\label{eq poisson problem}
\begin{pde}
\Delta u &= \Delta U &\text{in }\Omega\\
u &= g &\text{on }\partial\Omega
\end{pde}
\end{equation}
has a unique weak solution $u^*\in g + H_0^1(\Omega) \subset H^1(\Omega)$.
\item Let $\Omega = (0,a)\times (0,b)$ be a rectangular domain in $\R^2$ and $g, U, u^*$ as above. On the positive side we note:
\begin{itemize}
\item $u^*\in L^\infty(\Omega)$ and $\|u^*\|_{L^\infty(\Omega)} \leq \big(\|g\|_\B+2\,\|U\|_\B\big)(1+ \sqrt{a^2+b^2})$.
\item $u^*\in W^{1,q}(\Omega)$ and $\nabla u^* \in L^q(\partial\Omega)$ for all $1\leq q<\infty$. More quantitatively:
\[
\|\nabla u^*\|_{L^q(\Omega)} \lesssim \big( \|U\|_\B+ \|g\|_\B\big)q  \qquad\text{and}\quad \|\nabla u^*\|_{L^q(\partial\Omega)}\lesssim \big( \|U\|_\B+ \|g\|_\B\big) q  .
\]
\item  The `harmonic part' $u^*-U$ of the solution satisfies $(u^*-U) \in W^{2,p}(\Omega)$ for all $1\leq p<2$ and $\|D^2u^*\|_{L^p(\Omega)}\lesssim \|g\|_\B\,(2-p)^{-1/p}$.
\end{itemize}
On the negative side:
\begin{itemize}
\item Generally, $u^*\not\in \B$. This is true even if $\Delta u = \Delta U \equiv 0$.
\item The conditions $p<2$ and $q<\infty$ are necessary: In general, $u^*$ is neither in $W^{1,\infty}(\Omega)$ nor in $H^2(\Omega)$. This is true even if $\Delta u = \Delta U\equiv 0$. 
\end{itemize}
The constants hidden in the notation $\lesssim$ may depend on $a,b$, but not on $U, g, p, q$.
\item Under the same assumptions, for every $\eps\in(0,1/2)$, there exists $u_\eps\in \B$ such that $u_\eps \equiv g$ on $\partial\Omega$ and
\begin{align*}
\|u_\eps\|_\B &\lesssim \big( \|U\|_\B+ \|g\|_\B\big)\,|\log \eps|\\
  \|u_\eps - u^*\|_{L^\infty(\Omega)} &\lesssim \big( \|U\|_\B+ \|g\|_\B\big) \,\eps\\
    \|\nabla(u_\eps - u^*)\|_{L^1(\Omega)} &\lesssim \big( \|U\|_\B+ \|g\|_\B\big)\,\eps^{2}|\log\eps| \\
  \|\nabla(u_\eps - u^*)\|_{L^q(\Omega)} &\lesssim \big( \|U\|_\B+ \|g\|_\B\big)\left(\frac1{q-1}+q\right)\,\eps^{2/q} &&\forall\ q\in (1,\infty)\\
    \|\nabla (u_\eps -u^*)\|_{L^q(\partial\Omega)} &\leq \big( \|U\|_\B+ \|g\|_\B\big)\,q\,\eps^{1/q} &&\forall\ q\in [1,\infty)\\
\|D^2(u_\eps - u^*)\|_{L^{p}(\Omega)} &\lesssim \big( \|U\|_\B + \|g\|_\B\big)\,(2-p)^{-1/p} \,\eps^{\frac{2-p}p} && \forall\ p\in [1,2).
\end{align*}
The constants hidden in the $\lesssim$ notation may depend on $a, b$ but not on $U,g, p, q$ or $\eps$.

\item Let $\Omega = (0,a)\times(0,b)$ be a rectangular domain and $U, g\in\B$. Then for every $\eps\in (0,1/2)$, there exists $\tilde u_\eps \in \B\cap \big(U+ H^2(\Omega)\big)$ such that $\Delta u_\eps \equiv \Delta U$ and 
\begin{align*}
\|\tilde u_\eps\|_\B &\lesssim \big( \|U\|_\B+ \|g\|_\B\big)\,|\log \eps|\\
\|\tilde u_\eps - u^*\|_{L^\infty(\Omega)} &\lesssim \big( \|U\|_\B+ \|g\|_\B\big) \,\eps|\log\eps|\\
\|\tilde u_\eps - u^*\|_{L^q(\Omega)} &\lesssim \big( \|U\|_\B+ \|g\|_\B\big)\,q \eps &&\forall q\in[1,\infty)\\
\|\tilde u_\eps - u^*\|_{L^q(\partial \Omega)} &\lesssim \big( \|U\|_\B+ \|g\|_\B\big)\,q \eps &&\forall q\in[1,\infty)\\
\|\nabla(\tilde u_\eps - u^*)\|_{L^q(\Omega)} &\lesssim \big( \|U\|_\B+ \|g\|_\B\big) \frac1{2-q}\,\eps &&\forall\ q\in [1,2)\\
\|\nabla(\tilde u_\eps - u^*)\|_{L^2(\Omega)} &\lesssim \big( \|U\|_\B+ \|g\|_\B\big) \eps\sqrt{|\log\eps|} &&\\
\|\nabla(\tilde u_\eps - u^*)\|_{L^q(\Omega)} &\lesssim \big( \|U\|_\B+ \|g\|_\B\big) \left(q+ \frac1{q-2}\right)\,\eps^{2/q} &&\forall\ q\in (2,\infty)\\
 \|\nabla (\tilde u_\eps -u^*)\|_{L^1(\partial\Omega)} &\lesssim \big( \|U\|_\B+ \|g\|_\B\big)\,\eps|\log\eps|\\
 \|\nabla (\tilde u_\eps -u^*)\|_{L^q(\partial\Omega)} &\lesssim \big( \|U\|_\B+ \|g\|_\B\big)\,\left(q+ \frac1{q-1}\right)\,\eps^{1/q} &&\forall\ q\in (1,\infty)\\
\|D^2(\tilde u_\eps - u^*)\|_{L^{p}(\Omega)} &\lesssim \big( \|U\|_\B + \|g\|_\B\big)\,(2-p)^{-1/p} \,\eps^{\frac{2-p}p} && \forall\ p\in [1,2).\\
\|D^2(\tilde u_\eps-u^*)\|_{L^2(\Omega)} &\lesssim \big( \|U\|_\B+ \|g\|_\B\big)\,\sqrt{|\log \eps|}
\end{align*}
Note that the final estimate {\em grows} as $\eps\to 0^+$.
\end{enumerate}
\end{theorem}

The first point and the first observations in the second point are positive: If the right hand side and boundary condition to the PDE satisfy reasonable regularity conditions, a unique solution $u^*$ to the Poisson equation exists, and it satisfies certain regularity conditions. On the other hand, the final observations of the second point are negative: Neither are the second derivatives of $u^*$ -- which comprise the Laplacian -- square integrable, nor is the solution a Barron function (or even Lipschitz continuous). This is even true for harmonic functions with Barron space boundary values.

The third point is more positive: At least on a very simple domain, we can find a function of low Barron norm which attains the boundary condition exactly and `almost' satisfies the differential equation in the sense that $\Delta u_\eps$ is small in $L^q$ for $q<2$ (since $u_\eps$ and $u^*$ are $L^q$-close). The fourth point is similar, but considers the situation where the partial differential equation is satisfied exactly and the boundary values are attained only approximately.
Note that $u^* = (u^*-U) + U$ can be approximated by Barron functions in $W^{2,p}(\Omega)$ for $p<2$ even if $u^*\notin W^{2,p}$ since only the harmonic part $u^*-U$ has to be approximated. Similarly, if $\Delta U\in L^2(\Omega)$, elliptic regularity shows that $U\in H^2(\Omega)$ and the statement in the fourth point becomes
\[
\|\tilde u_\eps\|_{H^2(\Omega)} \lesssim \big(\|U\|_\B + \|g\|_\B + \|\Delta U\|_{L^2(\Omega)}\big) \sqrt{|\log\eps|}.
\]
In the greater context of scientific machine learning, the message of Theorem \ref{theorem main} is that even for very simple linear PDEs which behave well on periodic or unbounded domains, the presence of boundaries may lead to unexpectedly subtle behavior.

While we carefully consider the dependence on function classes $W^{1,q}$ and $W^{2,p}$ and on the small parameter $\eps>0$, we leave the dependence of constants on the aspect ratio $b/a$ of the domain $\Omega$ for future work. We also do not claim optimality in the construction of $u_\eps$ or in the rates of convergence and do not provide lower bounds on approximability.

Theorem \ref{theorem main} structurally resembles regularity theory in Lipschitz spaces, where the principle of maximal regularity is known to fail (i.e.\ solutions may be less regular than the data would in principle allow for). For instance, the `local Lipschitz constant' of a harmonic function $h$ with boundary values in $C^{0,1}(\partial\Omega)$ blows up logarithmically at the boundary if $\partial\Omega$ is $C^2$-regular \cite{hile1999gradient}: {\em If $\Delta h \equiv 0$ in $\Omega$ and $h|_{\partial\Omega}$ is Lipschitz-continuous, then}
\[
|\nabla h(x)|\leq C(\Omega)\,[h]_{C^{0,1}(\partial\Omega)} \,\log\left(\frac{2\,\diam(\Omega)}{\dist(x,\partial\Omega)}\right)\qquad\forall\ x\in\Omega.
\]
This exact blow-up behavior is observed for the simplest possible Barron boundary values $\sigma(w\cdot x+b)$ as seen below. The work of \cite{hile1999gradient} can be used to establish that $u^*, h\in W^{1,p}(\Omega)$ for all $p<\infty$ with linear norm growth (in $p$) for $C^2$-domains $\Omega$ and harmonic functions with Lipschitz boundary values (in particular, Barron boundary values). We pursue a more explicit construction since rectangular domains do not have smooth boundaries, and since it is crucial for the approximation theory by Barron functions, in particular approximation of gradients on the boundary.

As an application of Theorem \ref{theorem main}, we derive a priori error estimates for Deep Ritz neural PDE solvers when finding harmonic functions on a two-dimensional rectangular domain with Barron boundary values.
The Deep Ritz method builds on the variational characterization of the solution of \eqref{eq poisson equation intro} as the minimizer of the functional
\[
E_{DR} (u) =  \begin{cases} \int_\Omega \frac12\,\|\nabla u\|^2 - fu\dx &\text{if $u=g$ on $\partial\Omega$}\\ +\infty&\text{else,}\end{cases}
\]
where the PDE arises as the Euler-Lagrange equation (first order optimality condition) for $E_{DR}$.
For Deep Ritz solvers, the functional is discretized by replacing integrals with finite sums over evaluation points, boundary values enforced weakly by penalty, and general functions $u$ by neural networks $u(\theta,\cdot)$ of fixed architecture:
\[
\widehat E_{DR, N, n, \lambda}(\theta) = \frac{1}{N}\sum_{i=1}^N \left(\frac12\,\|\nabla_xu(\theta;x_i)\|^2 - f(x_i) \,u(\theta;x_i)\right) + \lambda\,\frac{1}{n} \sum_{j=1}^n \big|u(\theta; z_j) - g(z_j)\big|^2
\]
where $\theta$ comprises the weights and biases of the neural network and $x_i, z_j$ are integration points inside the domain and on the boundary respectively. We consider the functional in a class of shallow ReLU networks
\[
u(\theta; x) = u(a,W,b; x) =  \sum_{i=1}^m a_i \,\sigma(w_i\cdot x+b_i), \qquad (a,W,b) \in \R^m \times \R^{m\times 2} \times \R^m
\]
and add a weight decay regularizer
\[
\widehat E_{DR, N, n, \lambda, \mu}(a,W,b) = \widehat E_{DR, N, n, \lambda}(a,W,b) + \frac\mu2\big(\|a\|_2^2+ \|W\|_F^2 + \|b\|_2^2\big)
\]
to control the generalization error. Then the following is true.

\begin{theorem}\label{theorem deep ritz}
Let $\Omega\subseteq \R^2$ be a rectangular domain, $f\equiv 0$ and $g\in\B$. Assume that the integration points $x_i, z_j$ are sampled independently and uniformly in $\Omega$ and on $\partial\Omega$ respectively.
Assume further that $\delta\in(0, 1/4)$ and that the parameters $N, n, \lambda, \mu$ scale with the expressivity $m$ of the function class such that
\[
\lambda_m = \frac{\sqrt{m}}{\log m}, \qquad \mu_m = \frac{\gamma}{\sqrt m}, \qquad n, N\gg m^2
\]
with $\gamma \geq \frac{\|g\|_\B+1}{\sqrt m}$.
Then, with probability at least $1-2\delta$ and for sufficiently large $m\in\mathbb N$, if $(a,W,b)$ minimizes $\widehat E_{DR; N, n, \lambda,\mu}$, we have
\[
\big\|u(a,W,b;\cdot) - u^*\big\|_{H^1(\Omega)}^2 \lesssim \big(\|g\|_\B^2+1\big) \,\sqrt{\log(1/\delta)}\,\frac{\log m}{\sqrt m}
\]
where the hidden constant depends only on the rectangle $\Omega$ and encompasses errors from the discretization of the integral.
\end{theorem}

The same statement holds for deterministic choices of integration points as long as empirical averages converge to integrals at the Monte-Carlo rates $n^{-1/2}, N^{-1/2}$ uniformly over Barron functions on $\partial\Omega$ and indicator functions of the intersection of half-spaces inside $\Omega$. Faster rates of convergence can reduce the number of sample points required. 

The Monte-Carlo rate of convergence in $n, N$ is commonly found in Barron space a priori error estimates whereas in $m$, estimates often depend on $1/m$ or $\log m/m$ \cite{E:2018ab, park2023minimum}. The main obstruction in the deep Ritz setting is that boundary values are not enforced exactly, but only by penalty. For finite $\lambda$, the minimizer of the problem is {\em not} the solution to the PDE \eqref{eq poisson equation intro}, but a harmonic function with an $L^2$-close boundary condition which may have lower Dirichlet energy. The choice of $\lambda_m$ balances between enforcing boundary values exactly and allowing for spatial and function space discretization.

The Barron function space, in which we are considering boundary values and searching for solutions, is primarily well-understood in one dimension. Here $u \in \B(a,b)$ if and only if $u''$ is a (signed) Radon measure such that 
\begin{equation}\label{eq barron condition 1d}
\int_{(a,b)} \sqrt{1+x^2}\,\d |u''| < +\infty
\end{equation}
where $\d|u''|$ denotes the integral with respect to the variation measure of $u''$. If $u''\in L^1(a,b)$, this is simply $|u''|\dx$. If $|a|, |b| <\infty$, this means that $\B(a,b)$ is the space of all functions whose first (distributional) derivative is in $BV(a,b)$. Even in one dimension, this is a stricter condition than the Lipschitz property, which can be formulated as $u' \in L^\infty(a,b)$ by Rademacher's Theorem. Barron space more closely resembles the Sobolev space $W^{2,1}$, but only requires a measure-valued second derivative rather than a function in $L^1$. In the following, we show that the characterization \eqref{eq barron condition 1d} carries over to the one-dimensional boundary of the square.

\begin{theorem}[Barron boundary data for rectangular domains]\label{theorem boundary}
Let $\Omega= (0,a)\times (0,b)$ be a rectangular domain in $\R^2$ for $a,b<\infty$ and $g:\partial\Omega\to\R$ a continuous function. Then there exists $G\in \B(\R^2)$ such that $G(x) = g(x)$ for $x\in\partial\Omega$ if and only if the functions
\[
g(x,0), \quad g(x,b) \text{ are in } \B(0,a)\qquad \text{and}\quad g(0,y), \:\: g(a,y) \in\B(0,b).
\]
%i.e.\ if and only if
%\[
%\int_0^a \big|\partial_{xx} \,g(x,0)\big| + \big|\partial_{xx}\,g(x, b)\big|\dx + \int_0^b \big|\partial_{yy}\, g(0,y)\big| + \big|\partial_{yy}\,g(a,y)\big|\dy < + \infty.
%\]
\end{theorem}

The continuity condition on $g$ is a matching condition between the four functions of one variable at the corners of the rectangle. Since Barron functions are continuous, it is not merely sufficient, but in fact necessary.
Theorem \ref{theorem boundary} characterizes the class of boundary data for which the results of Theorem \ref{theorem main} are relevant. It is large compared to classical regularity classes $C^{2,\alpha}$ for boundary values, but small compared to the class $H^{1/2}$ for weak solutions with $H^1$-regularity. It is not comparable to the class $H^{3/2}(\partial\Omega)$, which induces $H^2$-regular weak solutions: The Hilbert-Sobolev class requires less differentiability than the Barron class (a half derivative in place of a full derivative) but higher integrability (square integrability instead of integrability/measure regularity). 

To focus on the impact of the boundary condition, we deliberately left the characterization of the right hand side in $\Delta u = \Delta U$ somewhat vague. Since $\B\subseteq W^{1,\infty}\subseteq H^1$, we have $\Delta U \in H^{-1}$ by design, i.e.\ the Laplacian is a well-defined object in a classical Hilbert space. Using the fact that the second derivative of the ReLU activation function $\sigma$ is a Dirac measure and that Barron spaces are convex combinations of ridge function ReLU activations, it is also a measure. As a minimal example we note that every Barron function $f$ can be written as the Laplacian of a Barron function $U$ (in a compact domain) since the anti-derivative $\max\{0, z^3\}/6$ of the ReLU activation is a Barron function -- compare e.g.\ \cite[Section 2.2]{wojtowytsch2020some}. However, the class of admissible right hand sides is significantly larger and contains e.g.\ unbounded functions such as $(2-\gamma)(1-\gamma) \,|x_1|^{-\gamma} = \Delta |x_1|^{2-\gamma}$ for $\gamma <1$ as well as certain proper measures.

The solution of partial differential equations in Barron spaces has been studied in cases where the solution is guaranteed to be a Barron function \cite{wojtowytsch2020some} or when the solution is only assumed to be in a Sobolev space \cite{lu2021machine}. The second work also studies in detail the convergence of regularized Deep Ritz and PINN solutions. On the whole space, approximation results of the form $\|u-u_\eps\|_{H^1}\leq\eps$ have been obtained for the solution of uniformly elliptic PDEs with non-constant Barron space coefficients where $\|u_\eps\|_\B$ grows slightly faster than polynomially in $1/\eps$, comparable to $\exp(C\log^2(1/\eps))$. The solution of PDEs in the related `spectral Barron classes' with a homogeneous Neumann boundary condition on the unit cube in $d$ dimensions has been studied, for instance, in \cite{lu2022priori, chen2023regularity}. A rate of convergence to the solution of a PDE was obtained for a greedy algorithm approximation in \cite[Theorem 8]{siegel2023greedy} under the assumption that solutions lie in Barron space (or more generally, the span of a suitable dictionary), but without consideration for whether regularity of the data would imply regularity of solutions.

Partial extensions of this work have been obtained \cite{vaishampayan2024solving} for powers of the ReLU activation function, and a preliminary study on the implicit bias of PINNs with ReLU activation was conducted in this context in \cite{vaishampayan2024shallow}.

The article is structured as follows. In Section \ref{section review}, we review background material related to Sobolev spaces and Barron spaces. Sections \ref{section theorem main}, \ref{section deep ritz} and \ref{section boundary} are devoted to the proofs of Theorems \ref{theorem main}, \ref{theorem deep ritz} and \ref{theorem boundary} respectively. In Section \ref{section conclusion}, we briefly discuss the impact of our results. Some technical results are postponed to the appendices.

\section{Review of related works and concepts}\label{section review}

\subsection{Sobolev spaces and regularity theory}\label{section background sobolev}

We assume that the reader is familiar with the basic concepts and properties of Sobolev spaces, including spaces $H^s$ of fractional order $s$. Excellent introductions can be found for instance in \cite{MR2759829, dobrowolski2010angewandte, di2012hitchhikers, leoni2023first}. For an introduction to elliptic regularity theory, see e.g.\ \cite{grisvard2011elliptic,gilbarg2015elliptic}.

Fractional order Sobolev spaces play a role for us in two places: (1) as spaces which embed into Barron spaces and (2) as the natural trace spaces for (integer) Sobolev functions on domains. Namely, if the boundary of $\Omega$ is at least Lipschitz-regular, then the trace operator $T$ which restricts a continuous function $u$ on $\overline\Omega$ to a continuous function $u|_{\partial\Omega}$ on the boundary $\partial\Omega$ extends to a continuous operator $T: W^{1,p}(\Omega)\to W^{1-\frac1p, p}(\partial\Omega)$ \cite[Section 15.3 and Proposition 14.40]{leoni2017first}. Since fractional Sobolev spaces embed {\em compactly} into Lebesgue spaces with subcritical exponent $p< p^*$ (see \cite[Theorem 6.13]{leoni2023first} for the whole space setting), we find in particular that the trace operator is {\em compact} as an operator $T:W^{1,p}(\Omega)\to L^p(\partial\Omega)$ \cite[Corollary 9.16]{leoni2023first}.

Sobolev and H\"older spaces can be characterized as spaces of functions which are well approximated by polynomials on small scales. For H\"older spaces, this is essentially by definition: For instance, 
\[
u\in C^{0,\alpha}(\R^d) \qquad\Ra\quad  \inf_{c\in\R} \|u- c\|_{L^\infty(B_r(0))} \leq \|u- u(0)\|_{L^\infty(B_r(0))} \leq [u]_{C^{0,\alpha}(B_r(0))} r^\alpha
\]
for all $r>0$. Conversely, if the estimate $\inf_{c\in\R}\|u- c\|_{L^\infty(B_r(x))}\leq C\,r^\alpha$ holds with the same constant $C$ for all $x\in \R^d$ and $r>0$, then $u$ is H\"older continuous. The argument can be integrated: If $\mathcal P_k$ denotes the space of polynomials of degree at most $k$ in $d$ variables, then 
\[
u\in C^{m,\alpha}(\R^d) \quad \Ra \quad \inf_{p\in\mathcal P_m} \|u-p\|_{C^k(B_r(0))} \leq C_{m,d}\,[D^mu]_{C^{0,\alpha}(B_r(0))} \,r^{m - k + \alpha} \qquad \forall\ x\in \R^d, \:r>0.
\]
Note that only the H\"older constant of the tensor of highest order derivatives is needed on the right hand side since the polynomial $p$ can be used to match the first $m$ derivatives suitably. In one dimension, we have
\[
\inf_{p\in\mathcal P_{m+1}} \|u-p\|_{C^0(-r,r)} = \inf_{p\in\mathcal P_{m+1}} \max_{x\in[-r,r]} \big|u-p\big|(x) 
	\leq \max_{x\in[-r,r]} \left|u(x) -\sum_{i=0}^{m+1} \frac{u^{(i)}(0)}{i!}\,x^i\right|
\]
and by induction
\begin{align*}
\max_{x\in[-r,r]} \left|u(x) -\sum_{i=0}^{m+1} \frac{u^{(i)}(0)}{i!}\,x^i\right| 
	&\leq \max_{x\in[-r,r]}\int_0^x \left|u' - \sum_{i=0}^m \frac{u^{(i)}(0)}{i!}\,s^i\right|\ds\leq \int_0^r C\,s^{m+\alpha}\ds = C r^{m+1+\alpha}
\end{align*}
The case $k\geq 1$ is true trivially by the induction hypothesis. 
For Sobolev functions, the analogous result is the Bramble-Hilbert Lemma -- see \cite{bramble1970estimation, dupont1980polynomial}: If $u\in W^{m,p}(B_r(0))$ in dimension $d$, then for all $0\leq k\leq m$ we have
\[
\inf_{p\in \mathcal P_{m-1}} \|u- p\|_{W^{k,p}(B_r(0))} \leq C_{m,d}\,\|D^mu\|_{L^p(B_r(0))} \,r^{m-k}.
\]
Again, the right hand side only depends on the semi-norm which measures the magnitude of the highest order derivatives.

\subsection{Neural networks and Barron spaces} \label{section barron review}
In a similar spirit, Barron spaces are comprised of functions which can be approximated well (globally) by ReLU networks with a single hidden layer. Barron spaces or slight variations thereof are referred to variously as $\mathcal F_1$ \cite{bach2017breaking}, Barron space \cite{E:2018abpub, weinan2019lei, review_article,han2023class}, Radon-BV spaces \cite{parhi2021banach, binev2022optimal, devore2023weighted} and the variation spaces of the ReLU dictionary \cite{siegel2021characterization,siegel2019approximation,siegel2021optimal,siegel2021sharp} by different authors. We note that the term `Barron space' is ambiguous and is at times also used to refer to function classes characterized by a moment condition on the Fourier transform \cite{voigtlaender2022p, chen2023regularity} or Fourier coefficients \cite{lu2022priori}. The relationship between the `spectral' and `representational' Barron classes is explored e.g.\ in \cite{barron_new, barron_boundaries,wu2023embedding}. In this note, we will always consider the {\em representational} version which is tailored to ReLU neural networks with a single hidden layer.

Let $\sigma(z) = \max\{z,0\}$. Recall that a neural network with a single hidden layer can be represented as $f_m(x) = \frac1m\sum_{i=1}^m a_i\,\sigma(w_i\cdot x+b_i)$.\footnote{\ In practical applications, the scaling factor $1/m$ is often dropped.}\ A function $f:\R^d\to\R$ is called a {\em Barron function} if there exists a probability distribution $\pi$ on the parameter domain $\R\times \R^d\times\R$ such that
\[
f(x) = f_\pi(x):= \E_{(a,w,b)\sim\pi} \big[a\,\sigma(w\cdot x+b)\big]\quad\forall\ x\in\R^d, \qquad \E_{(a,w,b)\sim\pi}\big[|a|\,\sqrt{\|w\|^2+|b|^2}\big] < +\infty,
\]
i.e.\ Barron spaces are comprised of functions where the empirical average sum for $a\,\sigma(w\cdot x+b)$ is replaced by a general expectation. The moment bound on $\pi$ enforces the convergence of the integral and implies that a Barron function $f$ is Lipschitz-continuous. The distribution $\pi$ in the representation of $f$ is never unique \cite[Section 2.1]{barron_new}. The norm in Barron space is therefore defined using an infimum:
\[
\|f\|_\B = \inf\left\{\E_{\pi}\left[|a|\,\sqrt{\|w\|^2+|b|^2}\right] : \pi\text{ s.t. }f\equiv f_\pi\right\} 
	= \frac12\inf\left\{\E_{\pi}\big[|a|^2 + \|w\|^2+|b|^2\big] : \pi\text{ s.t. }f\equiv f_\pi\right\} .
\] 
The identity between different definitions arises from the positive homogeneity $\sigma(\lambda z) = \lambda\,\sigma(z)$ for all $\lambda>0$, $z\in\R$. It only applies to ReLU and leaky ReLU activation in this form, but powers of ReLU enjoy similar homogeneity properties.

The variability between the different kinds of spaces arises depending on whether the magnitude of bias variables is controlled in the Barron norm, or just the magnitude of weights \cite{wojtowytsch2022optimal}. Either choice is valid, yielding perhaps surprisingly different results -- see e.g.\ \cite{boursier2023penalising} and \cite[Proposition 2.1]{wojtowytsch2022optimal}. Here, we opt for the version in which the magnitude of the bias is controlled, but we anticipate that our results are stable under making the opposite choice.

Since $\sigma'' = \delta$ is a Dirac measure at the origin, we find that $\Delta u$ is a locally finite signed measure on $\R^d$ for any Barron function $u$. In fact, in one dimension, this fully characterizes the class of Barron functions as 
\[
\B(\R) = \left\{f: \R\to\R \:\bigg|\: f'\in BV(\R)\text{ and } \int_\R \sqrt{1+x^2} \d |f''|_x<+\infty\right\}
\]
where $|f''|$ is the variation measure of the (measure-valued) second distributional derivative of $f$ \cite[Section 1.8]{evans2015measure}. See e.g.\ \cite{barron_new, li2020complexity, boursier2023penalising} for details.

 In higher dimensions $\Delta\,\sigma(w\cdot +b) = \|w\|\cdot \H^{d-1}|_{\{x : w\cdot x+b=0\}}$ is a multiple of  the $d-1$-dimensional Hausdorff measure on a hyperplane. Since a Barron function $u$ is a superposition of single neurons, its Laplacian is equally a superposition of area measures on hyperplanes. In particular $|\Delta u|(A) = 0$ if $\H^{d-1}(A) =0$. Consequently, there are many measures $\mu$ which cannot arise as the Laplacian of a Barron function. In a similar spirit, we may argue that $\Delta u \in W^{-1,\infty}(\Omega) = (W^{1,1}_0(\Omega))^*$ for $u\in \B\subseteq W^{1,\infty}(\Omega)$, i.e.\ $\Delta u$ enjoys a certain level of regularity. A characterization in this spirit in the Radon domain is given in \cite{ongie2019function}.

Similarly, we can argue that the Hessian of a Barron function is a measure with values in the space of symmetric $d\times d$-matrices. In this sense, Barron functions are sufficiently regular to consider linear (and even some non-linear) first and second order PDEs in them. In particular, we note that the functional $E_{q, PINN}'(u) = \int_\Omega |\Delta u|^q\dx$
can be extended directly to the entire Barron space in a natural way for $q=1$, but not $q>1$. However, many Barron functions are smooth and in fact, all sufficiently smooth functions are Barron (see the final point in the following list of known results on Barron spaces).

\begin{itemize}
\item Barron spaces are Banach spaces, but neither reflexive nor separable \cite[Remark 4.7]{barron_new}.

\item Barron spaces are particularly popular in high-dimensional analysis since they possess three favorable properties:
\begin{enumerate}
\item They are `large' in the sense that every linear method of approximation suffers from the `curse of dimensionality' in them \cite[Equation (1.33)]{siegel2021sharp}: For any function space $V_m$ of dimension $m$, the Kolmogorov $m$-width of $V$ in $\mathcal B$ is
\[
\sup_{\|f\|_\B\leq 1} \min_{f_m\in V_m} \|f-f_m\|_{L^2((0,1)^d)} \gtrsim m^{-\frac{3}{2d}}.
\]
The bound indicates that Barron spaces contain a variety of different functions, which is suggestive of the fact that they may be suited for the solution of a large class of problems.
%We take this as a suggestion that they contain a sufficiently large variety of functions to study a diverse class of problems in as the spaces must contain a large variety of functions.

\item They are `small' from an estimation perspective in the sense that the unit ball in Barron space has low Rademacher complexity \cite[Theorem 6]{weinan2019lei}. This means that integrals can be estimated well by Monte-Carlo averages {\em uniformly} over the function class, a property which is crucial when trying to optimize integral energies based on access only to a finite set of samples, see e.g.\ \cite[Proposition 5.1]{E:2018abpub} or \cite[Appendices D and E.1]{park2023minimum} and \cite[Lemma 26.2]{shalev2014understanding} for general Rademacher complexity-based generalization. The trade-off between approximation power and learnability of Barron spaces is discussed in \cite{approximationarticle}.

\item They are `small' from an approximation perspective when utilizing ReLU networks with a single hidden layer: Let $\Omega\subseteq B_R(0)\subseteq \R^d$ be a bounded open set. Then for every $m\in\N$ there exist $m$ weight vectors $(a_i, w_i, b_i)$ with $\frac1m \sum_{i=1}^m |a_i|\big(\|w_i\|+|b_i|\big) \leq \|f\|_\B$ such that $f_m(x) := \frac1m\sum_{i=1}^m a_i\sigma(w_i\cdot x+b_i)$ satisfies one of the following three bounds:

\begin{enumerate}
\item $\|u-u_m\|_{L^2(\Omega)} \leq C\max\{R,1\}\,\sqrt{|\Omega|}\,\|f\|_\B m^{-1/2 - 3/(2d)}$ \cite[Proposition 1]{bach2017breaking},  \cite{matouvsek1996improved}, 
\item $\|u-u_m\|_{L^\infty(\Omega)} \leq C\max\{R,1\}\,\|f\|_\B d^{1/2}m^{-1/2}$ \cite[Theorem 12]{review_article}, or
\item $\|u-u_m\|_{H^1(\Omega)} \leq C\max\{R,1\}\,|\Omega|\,\|f\|_\B d^{1/2}m^{-1/2}$ \cite[Equation (1.16)]{siegel2021sharp}, \cite[Equation (1.5)]{wojtowytsch2020some}.
\end{enumerate}

The proof of the third bound is common in dictionary learning, based on the Maurey-Barron-Jones Lemma \cite[Lemma 1]{barron1993universal}. It is most well-known in the context of $L^2$-spaces. The proof applies the same way in other Hilbert spaces, such as $H^1$, or even type-II Banach spaces such as $L^p$ or $W^{1,p}$ for $p\in[2,\infty)$ -- see \cite[Section 2]{siegel2021sharp}. The same claim therefore remains true in $W^{1,p}(\Omega)$ for all bounded open sets $\Omega$ and $p<\infty$, with a constant depending on $p$.

The first estimate is a dimension-dependent improvement of the Maurey-Barron-Jones bound which exploits the H\"older-continuity in $L^2$ of the dictionary with respect to its parametrization. The second bound can be proved using Rademacher complexity in a fashion which exchanges the classical roles of function and evaluation point.
\end{enumerate}

\item Any Barron function allows a decomposition
\[
f(x) = f_0(x) + \sum_{i=1}^\infty f_i(P_ix- z_i)
\]
where  $f_0 \in \B(\R^d)\cap C^1(\R^d)$ and for every $i\geq 1$, $P_i$ is an orthogonal projection to $\R^{d_i}$ with $1\leq d_i\leq d$, $z_i\in \R^{d_i}$ and $f_i:\R^{d_i}\to\R$ is a Barron function which is $C^1$-smooth on $\R^{d_i}$ except at the origin \cite[Theorem 5.9]{barron_new}.

In heuristic terms, we can say that the singular set of a Barron function is a union of low-dimensional affine spaces and that the nature of the singularity does not change along the space. This structure theorem has profound consequences in the context of PDE analysis:
\begin{enumerate}
\item Let $\phi:\R^d\to\R^d$ be a $C^1$-diffeomorphism such that $f\circ \phi$ is a Barron function whenever $f$ is a Barron function. Then $\phi$ is affine linear \cite[Theorem 5.18]{barron_new}.
 
\item Assume that $f_1\in \B$ and $f_2\in C^\infty(\R^d)$. Then in general $f_1\cdot f_2\notin \B$ \cite[Remark 5.16]{barron_new}.
\end{enumerate}
In particular, common techniques such as straightening the boundary or using a partition of unity cannot be applied in Barron spaces. This fragility under slight changes motivates our focus on a very particular setting in this article.

\item Every Barron function is Lipschitz-continuous: $\B\embeds W^{1,\infty}(\Omega)$ for all bounded domains $\Omega$ \cite[Theorem 3.3]{barron_new} and $\nabla u\in BV(\Omega)$ for all bounded domains $\Omega$ \cite[Lemma 1]{barronlipschitz}. 

Since the Barron norm bounds the Lipschitz-constant, we find that $\B$ embeds compactly into $C^0(K)$ for all compact sets $K$ and into $L^p(\mu)$ for all measures $\mu$ on $\R^d$ with finite $p$-th moments. Since we have been unable to find a reference, a full proof is therefore given in the Appendix in Lemma \ref{compact embedding theorem}. A similar argument is used in \cite[Corollary E.2]{park2023minimum}.

\item On the other hand, all sufficiently smooth functions are Barron functions:
\begin{enumerate}
	\item Let $n=d+1$ if $d$ is odd and $n = d+2$ if $d$ is even. Then $W^{n,1}(\R^d)\embeds \B$ \cite[Corollary 1]{ongie2019function}. 
	
	\item Let $f\in H^s(\R^d)$ for $s>d/2+1$. Then for every compact set $K\subseteq \R^d$, there exists $F\in\B$ such that $f(x) = F(x)$ for all $x\in K$ -- for a more detailed statement and references, see the appendix.

\end{enumerate}
In particular, $C_c^\infty(\R^d)\subseteq \B(\R^d)$.
\end{itemize}

\section{Proof of Theorem \ref{theorem main}: Regularity theory with Barron data}\label{section theorem main}

We will prove Theorem \ref{theorem main} in several steps: Initially, we consider the Laplace equation locally on a half-space $\{x\in\R^2: x_2>0\}$. Here, we can find an exact solution with ReLU boundary data (Figure \ref{figure harmonic function}). Second, we move to a sector $\{x\in\R^2: x_1>0, x_2>0\}$. The boundary condition on the new portion of the boundary is enforced using the method of mirror charges (Figures \ref{figure quadrant 1} and \ref{figure quadrant 2}). Finally, we consider a rectangular domain, where we need a correction term in addition to the explicit contributions. We use regularity theory to show that the corrector is sufficiently smooth to be a Barron function.

In the following, we always assume that $\eps\in(0,1/2)$ to ensure that $1\lesssim |\log\eps|$.

\subsection{Half-spaces}

Let $H:=\R^2_+ =  \{x\in\R^2 : x_2>0\}$ and $\Omega_R = H\cap B_R(0)$, $\Omega:= \Omega_1$. We consider the problem of solving Laplace equation $-\Delta u =0$  in $H$ with boundary condition $u(x) = \sigma(x_1)$ on the {\em straight} part of the boundary (and no condition on the curved segment). For context, we first prove a result in classical function classes before we turn our attention to Barron space.

For any function class $X$ over $\Omega$, denote $X_\sigma := \{u\in X : u(x_1, 0) = \sigma(x_1)\}$, where the boundary equality is understood in the sense of traces for classes $X$ whose elements may be discontinuous.

\begin{lemma}\label{lemma sobolev regularity}
\begin{enumerate}
\item The function 
\begin{equation}\label{eq harmonic function}
u^*(x) = \frac1\pi \left(x_1\cdot \arctan\left(\frac {x_1}{x_2}\right) - \frac {x_2}2\,\log(\|x\|^2)\right) + \frac{x_1}2
\end{equation}
satisfies
\[
\lim_{x_2\to 0^+} u^*(x_1,0) = \sigma(x_1), \qquad \Delta u^*(x) = 0 \qquad \forall\ x\in \R^2_+.
\]

\item $u^*\in W^{1,q}_{loc}(\R^2_+)\cap W^{2,p}_{loc}(\R^2_+)$ for all $q<\infty$ and $p<2$ with
\[
\|u^*\|_{W^{1,q}(\Omega)} \lesssim q, \qquad \|u\|_{W^{2,p}(\Omega)} \lesssim \left(\frac{1}{2-p}\right)^\frac1p
\]

\item There exists no function $u$ on $\R^2_+$ with boundary values $\sigma(x_1)$ which is both harmonic and locally Lipschitz-continuous, but there exists a family of functions $u_\eps$ such that $u_\eps(x_1,0) = \sigma(x_1)$ and $u_\eps\to u^*$ in $W^{1,q}_{loc}(\R^2_+)\cap W^{2,p}_{loc}(\R^2_+)$ for all $q\in [1,\infty)$ and $p\in[1,2)$.

\item There is no function $u\in H^2_{loc}(\R^2_+)$ such that $u(x_1, 0) = \sigma(x_1)$.
\end{enumerate}
\end{lemma}

The function $u^*$ is visualized in Figure \ref{figure harmonic function}.

\begin{figure}
\includegraphics[width = .47\textwidth]{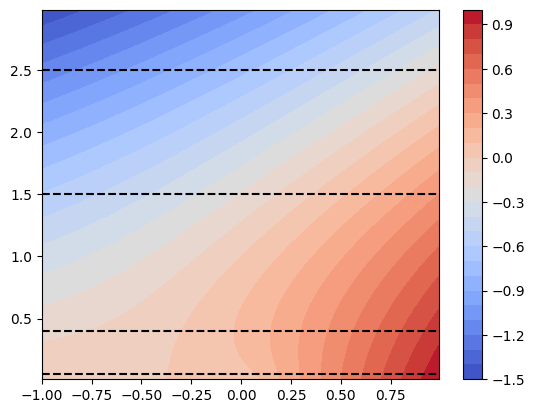}\hfill
\includegraphics[width = .49\textwidth]{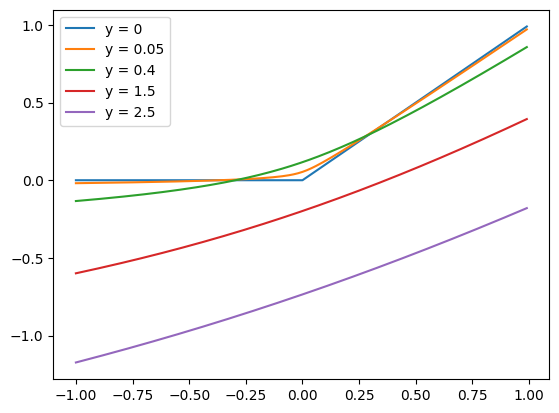}
\caption{\label{figure harmonic function}
We visualize the harmonic function described in Lemma \ref{lemma sobolev regularity} as a contour plot (left) and as a function of $x$ for various fixed values of $y$ (right). The lines along which we plot $u^*$ on the right are dashed in the contour plot on the left.\\
We note that $\partial_{x_1}u^*\to 1/2$ as $x_2\to \infty$, and indeed $u^*$ can be seen to converge to a linear function with slope $1/2$ as $x_2$ increases.
}
\end{figure}

\begin{proof}%[Proof of Lemma \ref{lemma sobolev regularity}]
{\bf Homogeneity and scaling.} 
The boundary condition $\sigma(x_1) = \frac1\lambda\,\sigma(\lambda x_1)$ is positively one-homogeneous, but the solution $u^*$ is not. If we rescale $u^*$, we obtain
\begin{align*}
\frac1\lambda \,u^*(\lambda x) &= \frac1\pi \left(\frac{\lambda x_1}\lambda \cdot \arctan\left(\frac {x_1}{x_2}\right) - \frac {\lambda x_2}{2\lambda }\,\log(\lambda^2\|x\|^2)\right) + \frac{\lambda x_1}{2\lambda}\\
	&= u^*(x) - x_2\,\frac{\log(\lambda)}\pi.
\end{align*}
The rescaled function is another harmonic function with boundary values $\sigma(x_1)$. In particular, we can use the homogeneity to obtain norm bounds easily:
\begin{align*}
\|u^*\|_{L^\infty(\Omega_R)} &= R\,\|u^*\|_{L^\infty(\Omega)} + R|\log R|\\
\|\nabla u^*\|_{L^q(\Omega_R)} &= \left(\int_{\Omega_R} \|\nabla u^*(x)\|^q \dx\right)^\frac1q
	= \left(\int_{\Omega_R} \left\|\nabla \left(R u^*\left(\frac xR\right) + x_2\log(1/R)\right)\right\|^q \dx\right)^\frac1q\\
	&\leq \left(R^2 \int_{\Omega_R} \|\nabla u^*\|^q(x/R)\,\frac1{R^2}\dx\right)^\frac1q + \big(R^2|\log(R)|\big)^\frac1q\\
	& = R^{2/q}\|\nabla u^*\|_{L^q(\Omega)} + R^{2/q}|\log R|^{1/q}\\
\|D^2u^*\|_{L^p(\Omega_R)} &= \left(\int_{\Omega_R}\left\|\frac1R\,D^2u^*\left(\frac xR\right)\right\|^p\dx\right)^\frac1p 
	\\&= R^{\frac{2-p}p}\left(\int_{\Omega_R} \left\|D^2u^*\left(\frac xR\right)\right\|^p \frac1{R^2}\dx\right)^\frac1p = R^{\frac{2-p}p} \|D^2u^*\|_{L^p(\Omega)}.
\end{align*}
In the following, we will focus on norm bounds and approximation on $\Omega = \Omega_1$, but the bounds could easily be generalized to different scales. The additive terms can be eliminated by choosing the solution $u_R = R\,u^*(x/R)$ instead of $u^*$ for different length scales.

{\bf First claim.} 
$u^*$ 
satisfies
\begin{equation}\label{eq boundary values attained}
\lim_{x_2\to 0}u^*(x) = \frac{x_1\cdot \lim_{z\to\infty}\arctan(x_1z)}\pi + \frac {x_1}2 = \frac{x_1\,\mathrm{sign}(x_1)+x_1}2 = \sigma(x_1).
\end{equation}
We compute its derivatives as
\begin{align*}
\pi\,\partial_{x_1}u^*(x) &= \arctan\left(\frac{x_1}{x_2}\right) + x_1\, \frac{1}{1+ (\frac{x_1}{x_2})^2}\,\frac1{x_2} - \frac {x_2}2\,\frac{2x_1}{\|x\|^2} + \frac\pi2
	%&= \arctan\left(\frac{x_1}{x_2}\right) + \frac{x_1x_2}{x_1^2+x_2^2}- \frac{x_1x_2}{x_1^2+x_2^2} + \frac12\\
	=  \arctan\left(\frac{x_1}{x_2}\right) + \frac\pi2\\
\pi\,\partial_{x_2}u^*(x) &= x_1\,\frac{1}{1+ (\frac{x_1}{x_2})^2}\,\left(-\frac{x_1}{x_2^2}\right) - \frac12\,\log(\|x\|^2) - \frac{x_2}2\,\frac{2x_2}{\|x\|^2}
	%&= \frac{-x_1^2}{x_1^2 + x_2^2} - \frac12\,\log\big(\|x\|^2\big) - \frac{x_2^2}{\|x\|^2}\\
	= -1 - \frac12\,\log\big(\|x\|^2\big)\\
\pi\,\partial_{x_1}^2 u^*(x) &= \frac{1}{1+ \big(\frac{x_1}{x_2}\big)^2}\,\frac1{x_2}
	= \frac{x_2}{\|x\|^2}\\
\pi\,\partial_{x_1}\partial_{x_2} u^*(x) &= -\frac12\,\frac{2x_1}{\|x\|^2} = \frac{-x_1}{\|x\|^2}\\ 
\pi\,\partial_{x_2}^2 u^*(x) &= \frac{-x_2}{\|x\|^2}.
\end{align*}
In particular
\[
\Delta u^* (x) = (\partial_{x_1}^2 + \partial_{x_2}^2)u^*(x) =  \frac1\pi \left(\frac{x_2}{\|x\|^2} +\frac{-x_2}{\|x\|^2}\right) =0\qquad\forall\ x\in \R^2_+.
\]

{\bf Second claim.}
Direct inspection shows that $u^*\in L^\infty_{loc}(\R^2_+)$ and $\|u^*\|_{L^\infty(\Omega_R)} \leq R(1 + \log(R)/\pi)$ for all $R\geq 1$. For all $q\in[1,\infty)$ and $p\in[1,2)$, we have
\begin{align*}
\|Du^*\|_{L^q(\Omega)} &\leq C\left(\int_{B_1(0)} \big(1+ \big|\log \|x\|\,\big|\big)^q\dx\right)^\frac1q = C \left(\int_0^1 r\big(1+|\log r|\big)^q\dr\right)^\frac1q < \infty\\
\|D^2u^*\|_{L^p(\Omega)} &\leq C\left(\int_{B_1(0)} \|x\|^{-p}\dx\right)^\frac1p = C\left(\int_0^1 r^{1-p} \dr\right)^\frac1p  \lesssim \left(\frac{1}{2-p}\right)^\frac1p< +\infty.
\end{align*}
More quantitatively, norm of the gradient can be estimated by
\begin{align*}
\left(\int_0^1 r(1+|\log r|)^q\dr\right)^\frac1q& \leq\left(\int_0^1 r\dr\right)^\frac1q + \left(\int_0^1 r|\log r|^q\dr\right)^\frac1q \leq 1 + q
\end{align*}
for large enough $q$ -- see Appendix \ref{appendix integrals} for details.

{\bf Third claim.} 
We first show that there exists no harmonic function $v\in W^{1,\infty}_\sigma(\Omega)$, i.e.\ no harmonic Lipschitz function with boundary values $\sigma(x_1)$. 
The function $u^*$ fails to be Lipschitz-continuous both at the origin and at infinity since
\[
\frac{u^*(0,x_2) - u^*(0,0)}{x_2} =  \frac12\,\log(x_2^2) = \log(x_2)\to \pm \infty \qquad\text{as }x_2\to \infty \text{ or }0^+ \text{ respectively.}
\]
Now, if such a function $v$ were to exist, we would necessarily have $v = u^*+h$ where $h:\Omega\to\R$ is a harmonic function with Dirichlet boundary values $h(x_1,0) \equiv 0$. Any such function $h$ extends to a harmonic function on $B_1(0)$ by odd reflection: $h(x_1,x_2) = -h(x_1, -x_2)$ for $x_2\leq 0$ (Schwarz reflection principle, see \cite[Problem 2.4]{gilbarg2015elliptic}). Consequently $h$ is infinitely smooth and in particular Lipschitz-continuous at the origin. 
Since $h$ is Lipschitz-continuous and $u^*$ is not, also $v$ must fail to be Lipschitz-continuous at the origin, contradicting our original assumption.

It remains to show that $u^*$ can be approximated in $W^{1,q}$ and $W^{2,p}$ by Lipschitz-functions with the correct boundary values. 
Modifying Step 1, we consider
\begin{equation}\label{eq ueps}
u_\eps(x) := \frac1\pi \left(x_1\cdot \arctan\left(\frac {x_1}{x_2}\right) - \frac{x_2}2\,\log(\|x\|^2+\eps^2)\right) + \frac{x_1}2.
\end{equation}
As above, we see that $u_\eps(x_1,0) = \sigma(x_1)$ for all $\eps>0$. Again, we compute the rescaled version
\begin{align*}
\frac1\lambda \,u_\eps(\lambda x) &= \frac1\pi\left(\frac{\lambda x_1}{\lambda} \,\arctan\left(\frac{\lambda x_1}{\lambda x_2}\right)- \frac{\lambda x_2}{2\lambda} \,\log\big(\lambda^2\|x\|^2 + \eps^2\big)\right) + \frac{\lambda x_1}{2\lambda}\\
	&= \frac1\pi\left( x_1\,\arctan\left(\frac{x_1}{x_2}\right) - \frac{x_2}2 \,\log\big(\|x\|^2 + \eps^2/\lambda^2\big) + \frac{\log(\lambda^2)}2\,x_2\right)+ \frac{x_1}2\\
	&= u_{\eps/\lambda}(x) + \frac{\log\lambda}\pi\,x_2
\end{align*}
The derivatives of $u_\eps$ are
\begin{align*}
\pi\,\partial_{x_1}u_\eps(x) &= \arctan\left(\frac{x_1}{x_2}\right) + \frac{x_1x_2}{x_1^2+x_2^2} - \frac{x_1x_2}{\|x\|^2 +\eps^2} + \frac\pi2\\
	%&&= \arctan\left(\frac{x_1}{x_2}\right) + \frac{\eps^2 \,x_1x_2}{\|x\|^2(\|x\|^2+\eps^2)} + \frac\pi2\\
\pi\,\partial_{x_2}u_\eps(x) &= \frac{-x_1^2}{x_1^2+x_2^2} - \frac12\,\log(\|x\|^2+\eps^2) - \frac{x_2^2}{\|x\|^2+\eps^2}
	%&&= -1 + \frac{\eps^2}{\|x\|^2(\|x\|^2+\eps^2)} - \log(\|x\|^2+\eps^2).
\end{align*}
As $2|x_1x_2| \leq x_1^2 +x_2^2$ and $|\log(\eps^2)|<+\infty$, the first derivatives of $u_\eps$ are bounded, i.e.\ $u_\eps$ is Lipschitz-continuous with Lipschitz-constant at most $C(1+|\log\eps|)$. The second derivatives of $u_\eps$ are
\begin{equation}\label{eq second derivatives}
\begin{split}
\pi\,\partial_{x_1}\partial_{x_1}u_\eps(x) &= \frac{1}{x_2}\,\frac1{1+(x_1/x_2)^2} + \frac{x_2}{x_1^2+x_2^2} - 2x_1\,\frac{x_1x_2}{\|x\|^4} - \frac{x_2}{\|x\|^2+\eps^2}+2x_1\,\frac{x_1x_2}{(\|x\|^2+\eps^2)^2}\\
	&= 2\left(\frac{x_1^2x_2}{(\|x\|^2 +\eps^2)^2} - \frac{x_1^2x_2}{\|x\|^4} \right) + 2\,\frac{x_2}{\|x\|^2}- \frac{x_2}{\|x\|^2+\eps^2}\\
\pi\,\partial_{x_1}\partial_{x_2}u_\eps(x) &= \frac{-2x_1}{\|x\|^2} +\frac{2x_1^3}{\|x\|^4} - \frac{x_1}{\|x\|^2+\eps^2} +\frac{2x_2^2x_1}{(\|x\|^2+\eps^2)^2}\\
\pi\,\partial_{x_2}\partial_{x_2}u_\eps(x) &= \frac{2x_1^2x_2}{\|x\|^4} - \frac{x_2}{\|x\|^2+\eps^2} - \frac{2x_2}{\|x\|^2+\eps^2} + \frac{2\,x_2^3}{(\|x\|^2+\eps^2)^2}.
\end{split}
\end{equation}
i.e.\ in particular 
\begin{equation}\label{eq laplacian ueps}
\pi \Delta u_\eps(x) = 2\,\frac{\|x\|^2\,x_2}{(\|x\|^2+\eps^2)^2} + \frac{x_2}{\|x\|^2}- 3\,\frac{x_2}{\|x\|^2+\eps^2}.
\end{equation}
Clearly $\Delta u_\eps\to 0 \equiv \Delta u$ pointwise and $|\Delta u_\eps(x)|\leq 6/\|x\| =: \Psi(x)$ for all $x\in \R^d$. Since $\Psi \in L^p(\Omega)$ for $p<2$, we find that $\Delta u_\eps \to 0$ in $L^p(\Omega)$ by the dominated convergence theorem. The fact that $u_\eps \in W^{2,p}$ with a uniform bound follows by the same argument.

{\bf Fourth claim.} We will show by contradiction that there exists no function $u\in H^2(\Omega)$ such that $u(x_1,0) = \sigma(x_1)$ for $x_1\in(-1,1)$. A fortiori, there exists no such $u\in W^{2,p}(\Omega)$ for any $p\geq 2$. 

Assume that $u\in H^2(\Omega)$ such that $u(x_1, 0) = \sigma(x_1)$. Since $\partial\Omega\in C^{0,1}$, we may extend $\partial_{x_1}u \in H^1(\Omega)$ to a function $v\in H_0^1(\R^2)$ by \cite[Satz 6.11]{dobrowolski2010angewandte} (see also \cite[Section 9.2]{MR2759829}). Notably, $v$ satisfies $v(x_1,0) = 1_{(0,\infty)}(x_1)$ in the sense of traces for $x_1\in(-1,1)$. Due to the jump singularity at the origin, we have
\begin{align*}
[v]_{H^{1/2}(-1,1)}^2 &= \int_\R\int_\R\frac{|v(x_1,0) - v(x_1',0)|^2}{|x_1-x_1'|^2}\dx_1\dx_1' \geq \int_{-1}^1\int_{-1}^1\frac{|1_{(0,\infty)}(x_1) - 1_{(0,\infty)}(x_1')|^2}{|x_1-x_1'|^2}\dx_1\dx_1'\\
	& = 2\int_{0}^1\int_{-1}^0 \frac1{|x_1-x_1'|^2}\dx_1'\dx_1 = 2\int_0^1 \int_0^1 \frac1{(t+s)^2}\dt\ds = 2\int_0^1\frac{1}t - \frac1{t+1}\dt = +\infty
\end{align*}
i.e.\ $v(x_1,0)\notin H^{1/2}(\partial \R^2_+)$. This contradicts the assumption that $v\in H^1(\R^2)$ by \cite[Proposition 14.40 and Theorem 15.20]{leoni2017first} and thus the assumption that $u\in H^2(\Omega)$ (alternatively, see \cite[Satz 9.40]{dobrowolski2010angewandte}).
\end{proof}

\begin{remark}\label{remark normal derivative}
We note that the gradient on the boundary is 
\[
\nabla u^*(x_1,x_2) = \big(1_{\{x_1>0\}}, \: -\pi^{-1}(1+ \log |x_1|)\big) \in L^p_{loc}(\partial \R^2_+)
\]
for all $p<\infty$. Tracing this through the following lemmata, we see that if $u$ is a harmonic function with Barron boundary data on a rectangular domain in $\R^2$, then $\partial_\nu u\in L^2(\partial \Omega)$.
\end{remark}

To us, the initial significance of Lemma \ref{lemma sobolev regularity} is that we cannot understand the Barron regularity of solutions to Poisson equations with Barron boundary values immediately from known function space embeddings. In dimension $d=2$, all function spaces which are known to embed into Barron space -- namely $W^{4,1}$ due to \cite[Corollary 1]{ongie2019function} and $H^s$ for $s>(d+2)/2$ due to \cite{barron_boundaries} -- also embed into $H^2$, i.e.\ they cannot match the boundary values. In fact, the harmonic function with ReLU boundary values is not Lipschitz-continuous, and thus not a Barron function. However, it is evident that the boundary values {\em can} be matched in Barron space.
%Thus, function space embeddings do not directly guarantee the existence of even a finite energy Barron function for a PINN functional, let alone one of near optimal energy. On the other hand, 
The proof of Lemma \ref{lemma sobolev regularity} gives us a road map for working in Barron spaces. We make the computations more quantitative in the sharper version below.

\begin{lemma}\label{lemma barron regularity}
Let  $u^*$ be
as in \eqref{eq harmonic function} and $u_\eps$ as in \eqref{eq ueps} in $\Omega$. Then there exist universal constants $C$ such that
\[
\|u_\eps\|_{\B(\Omega)} \leq C\,|\log\eps|,
\]
$ \|u_\eps-u^*\|_{L^\infty(\Omega)}\leq \eps/2$ and
\[
\|\nabla(u_\eps - u^*)\|_{L^{q}(\Omega)} \leq C\left(\frac1{q-1}+q\right)\,\eps^{2/q}, \qquad \|\nabla (u_\eps - u^*)\|_{L^q(B_1\cap \partial \R^2_+)} \leq Cq\,\eps^{1/q}
\]
for all $q\in(1,\infty)$ and
\[
\|D^2(u_\eps - u^*)\|_{L^{p}(\Omega)} \leq C\,(2-p)^{-1/p}\,\eps^{\frac{2-p}p}\qquad\forall\ p\in[1,2).
\]
\end{lemma}

Thus, while $u^*$ is not a Barron function itself, it can be approximated spectacularly well by functions of `low' Barron norm. We note that $u_\eps$ is {\em not} a Barron function (or even Lipschitz continuous) as it grows superlinearly as $x_2\to \infty$. However, as Lemma \ref{lemma barron regularity} asserts, there exists a Barron function $u_\eps'$ such that $u_\eps'\equiv u_\eps$ on $\Omega$ and the claims hold technically for $u_\eps'$. In the following, we will not distinguish between $u_\eps$ and $u_\eps'$ in our notation.

As a consequence, we see that $\inf_{u\in\B_\sigma} \int_\Omega |\Delta u|\dx =0$, but a minimizer $u\in \B_\sigma := \{u\in \B: u(x_1,0) = \sigma(x_1)\}$ does not exist. The same proof shows that for $q=1$, the estimate $\|\nabla (u_\eps-u^*)\|_{L^1(\Omega)} \leq C\eps^2|\log\eps|$ holds.

\begin{proof}
We decompose $u_\eps = \bar u + v_\eps$ where $v_\eps(x) = \frac{x_2}{2\pi}\,\log(\|x\|^2+\eps^2)$ is the logarithmic correction term.

{\bf Step 1: Boundary values.} For a Barron function $\Phi:\R\to\R$ with representation
\[
\Phi(z) = \E_{(a,w,b)\sim\pi}\big[a\sigma(wz +b)\big]
\]
we consider
\[
u_\Phi:\R^2_+\to\R, \qquad \bar u(x) = \E_{(a,w,b)\sim\pi}\big[a\sigma(wx_1 +bx_2)\big].
\]
By the positive homogeneity of the ReLU activation function $\sigma(z) = \max\{z,0\}$, we have
\[
 u_\Phi(x_1,x_2) =  \E_{(a,w,b)\sim\pi}\big[a\sigma(wx_1+bx_2)\big] =x_2\,\E_{(a,w,b)\sim\pi}\left[a\sigma\left(w\frac{x_1}{x_2}+b\right) \right] = x_2 \,\Phi\left(\frac{x_1}{x_2}\right)
\]
for $x_2>0$. Note that $\Phi(z) = \frac1\pi\,z\,\arctan(z)+\frac{z}2$ is a Barron function on the entire real line since $\Phi''(z) = \frac{2}{\pi\,(1+z^2)^2}$ has finite first moments \cite[Example 4.1]{barron_new}. In this case, we note that
\[
\bar u(x):= u_\Phi(x) = x_2\left(\frac1\pi\,\frac{x_1}{x_2}\,\arctan\left(\frac {x_1}{x_2}\right)+\frac{x_1}{2x_2}\right)  = \frac{x_1}\pi\,\arctan\left(\frac{x_1}{x_2}\right) + \frac{x_1}2
\]
as desired. Since Barron functions are continuous, the equality also holds for $x_2=0$.

{\bf Step 2: Logarithmic corrector.} The second term 
\[
v_\eps(x) := \frac{x_2}{2\pi}\,\log\big(\|x\|^2 +\eps^2\big)
\]
in the approximation $u_\eps = \bar u + v_\eps$ of $u^*$ given in \eqref{eq ueps} is infinitely smooth on $\R^2$. While $v_\eps$ grows to $\infty$ superlinearly at $\infty$ and cannot be a Barron function (or even a Lipschitz function), if we may multiply by an infinitely smooth and compactly supported cut-off function $\eta$ such that $\eta \equiv 1$ on $B_1(0)$, we can guarantee that there exists a Barron function $\tilde v_\eps :\R^d\to\R$ such that $\tilde v_\eps \equiv v_\eps$ on $B_1(0)$. We use the quantitative estimate
\begin{align*}
\|\tilde v_\eps\|_\B &\leq C s^{-1/2} \|\tilde v_\eps\|_{H^{2+s}(\R^d)} \leq C s^{-1/2}\, \|\tilde v_\eps\|_{H^{2}(\R^d)}^{1-s} \|\tilde v_\eps\|_{H^3(\R^d)}^{s} \leq Cs^{-1/2}\,\|v_\eps\|_{H^2(B_2(0))}^{1-s}\|v_\eps\|_{H^3(B_2(0))}^{s}
\end{align*}
for $s\in(0,1)$. For details, see Lemmas \ref{lemma sobolev embedding} and \ref{lemma sobolev interpolation} in the Appendix. The constant $C$ may change value from inequality to inequality, but does not depend on $s$. Thus, to obtain estimates, we may work with $v_\eps$ directly and choose $s$ to obtain a tight bound depending on $\eps$.

{\bf Step 2.1: Barron norm bound.} Considering only the terms corresponding to $v_\eps$ in \eqref{eq second derivatives}, we have
\begin{align*}
\pi\,\partial_{x_1}\partial_{x_1}v_\eps(x) &= 2\frac{x_1^2x_2}{(\|x\|^2 +\eps^2)^2} - \frac{x_2}{\|x\|^2+\eps^2}\\
\pi\,\partial_{x_1}\partial_{x_2}v_\eps(x) &=  - \frac{x_1}{\|x\|^2+\eps^2} +\frac{2x_2^2x_1}{(\|x\|^2+\eps^2)^2}\\
\pi\,\partial_{x_2}\partial_{x_2}v_\eps(x) &=  -3\frac{x_2}{\|x\|^2+\eps^2} + \frac{2\,x_2^3}{(\|x\|^2+\eps^2)^2},
\end{align*}
We easily compute the $H^2$-bound
\begin{align*}
\int_\Omega \|D^2v_\eps\|^2\dx &\leq C \int_0^1 \left\{\left(\frac{r^3}{(r^2+\eps^2)^2}\right)^2 + \left(\frac{r}{r^2+\eps^2}\right)^2\right\}r\dr\\
	&\leq C\left(\int_0^\eps \frac{r^7}{\eps^8} + \frac{r^3}{\eps^4}\dr + \int_\eps^1 \frac1r\dr\right) \leq C\big(1+|\log\eps|),
\end{align*}
i.e.\ $\|v_\eps\|_{H^2(\Omega)} \leq C\sqrt{1+|\log\eps|} \leq C\,|\log\eps|^{1/2}$. To bound the $H^3$-norm of $v_\eps$, we choose representative derivatives of individual terms -- all other terms follow the same pattern. A simpler but less precise argument is the appeal to homogeneity: For $\eps=0$, the tensor of $k-th$ derivatives $D^kv_0$ is a positively homogeneous function of degree $1-k$. The norm blow up $\int_\Omega \|D^kv_\eps\|^2\dx$ is logarithmic for $k=2$, just as the blow-up of the integral $\int_\eps^1 r^{1+2(1-k)}\dr = \int_\eps^1 r^{3-2k}\dr$ would suggest. The same heuristic yields a blow-up $\int_\Omega \|D^3v_\eps\|^2\dx \sim \int_\eps^1\frac1{r^3}\dr \sim \eps^{-2}$.
\begin{align*}
\nabla\left( \frac{x_2}{\|x\|^2 +\eps^2}\right) &= -\frac{2x_2}{(\|x\|^2+\eps^2)^2} \,\big(x_1, x_2\big) + \frac1{\|x\|^2+\eps^2} \,(0,1)\\
\nabla \left(\frac{x_1^2x_2}{(\|x\|^2+\eps^2)^2}\right) &= - \frac{4\,x_1^2x_2}{(\|x\|^2+\eps^2)^3}\,(x_1,x_2) + \frac{1}{(\|x\|^2+\eps^2)^2}\,(2x_1x_2, x_1^2).
\end{align*}
As previously, we observe that
\begin{align*}
\int_\Omega \|D^3v_\eps\|^2\dx &\leq C\int_0^1 \left\{\left(\frac{r^2}{(r^2+\eps^2)^2}\right)^2 + \left(\frac{r^4}{(r^2+\eps^2)^3}\right)^2\right\}r\dr\\
	&\leq C\left(\int_0^\eps \frac{r^5}{\eps^8} + \frac{r^9}{\eps^{12}}\dr +\int_\eps^1 \frac{1}{r^3}\dr \right) \leq C\left(1+ \eps^{-2}\right)
\end{align*}
and thus that $\|v_\eps\|_{H^3} \leq C\,\eps^{-1}$.
Choosing $s = s_\eps=1/ |\log\eps|$ in the interpolation between $H^2$ and $H^3$ implies the estimate
\[
\|v_\eps\|_\B \leq C\,s_\eps^{-1/2} |\log\eps|^{(1-s_\eps)/2} \eps^{-s_\eps} \leq C\,|\log\eps|^{1/2}\,|\log\eps|^{1/2} \,\exp\big(-\log\eps / |\log\eps|) = C\,|\log\eps|.
\]

{\bf Step 2.2: $L^\infty$-estimate.}
 We note that 
\[
\|u_\eps - u^*\|_{L^\infty(\Omega)} = \max_{x\in\overline \Omega} \left|\frac {x_2}2 \,\log\left(1+ \frac{\eps^2}{\|x\|^2}\right) \right|  = \max_{x_2\in[0,1]} \frac {x_2}2 \,\log\left(1+ \frac{\eps^2}{x_2^2}\right).
\]
The maximum is achieved either at $x_2=1$ with an objective value of $\log(1+\eps^2)/2 \leq \eps^2/2$ or at $x_2$ satisfying
\begin{align*}
0 &=\frac{d}{dx_2} \left(x_2\log\left(1+ \frac{\eps^2}{x_2^2}\right)\right) = \log(1+\eps^2/x_2^2) + \frac{x_2}{1+\eps^2/x_2^2}\cdot \left(-2\,\frac{\eps^2}{x_2^3}\right) \\
	&= \log\left(\frac{x_2^2+\eps^2}{x_2^2}\right) -2 \frac{\eps^2}{x_2^2+\eps^2}.
\end{align*}
This means that $z:= x_2^2/(x_2^2+\eps^2)$ solves $\log(1/z) = 2(1-z)$. The solution of this equation is $z^*\approx 0.2$ and we can solve 
\[
1- z^* = \frac{\eps^2}{x_2^2+\eps^2} \quad\Ra\quad (1-z^*)(x_2^2+\eps^2) = \eps^2 \quad \Ra \quad x_2 = \sqrt{\frac{z^*}{1-z^*}}\,\eps.
\]
It follows that
\[
\|u_\eps - u^*\|_{L^\infty(\Omega)} = \frac12\,\log(1/z^*) \,\sqrt{\frac{z^*}{1-z^*}}\eps\: \leq\, \frac{\eps}2.
\]

{\bf Step 2.3: $W^{1,q}$-estimate.}
Observe that 
\[
\nabla u_\eps - \nabla u^* = \frac1\pi \left(\frac{x_1x_2}{\|x\|^2} - \frac{x_1x_2}{\|x\|^2+\eps^2}, \: \frac12\left(\log(\|x\|^2) - \log(\|x\|^2+\eps^2)\right) + \frac{x_2^2}{\|x\|^2} - \frac{x_2^2}{\|x\|^2+\eps^2}\right).
\]
We compute
\begin{align}\notag \label{eq w1q estimate 2}
\int_\Omega \big|\log(\|x\|^2+&\eps^2) - \log(\|x\|^2)\big|^q\dx =\pi \int_0^1 \log^q\left(1+ \frac{\eps^2}{r^2}\right) r\dr\\
& = \pi\eps^2 \int_0^1 \log^q\left(1+ \left(\frac \eps r\right)^2\right)\,\left(\frac r\eps\right)^3\,\frac{\eps}{r^2}\dr = \pi\eps^2 \int_\eps^\infty \frac{\log^q(1+z^2)}{z^3}\dz.
\end{align}
Recall that $q>1$. The integral can be estimated as
\begin{align*}
\int_\eps^\infty \frac{\log^q(1+z^2)}{z^3}\dz 
%	&= \int_\eps^1 \frac{\log^q(1+z^2)}{z^3}\dz + \int_1^\infty \frac{\log^q(1+z^2)}{z^3}\dz\\
%	&\leq \int_\eps^1 z^{2q-3}\dz + \int_1^\infty\frac{\log^q(2z^2)}{4z^4}\,4z\dz\\
%	&\leq \frac{1}{2(q-1)} + \int_2^\infty\frac{\log^q(\xi)}{\xi^2}\d\xi\\
%	&\leq \frac{1}{2(q-1)} + \int_0^\infty \frac{\log^q(e^t)}{e^{2t}}e^t\dt  \\
%	&= \frac{1}{2(q-1)} + \int_0^\infty t^qe^{-t}\dt\\
	&= \frac1{2(q-1)} + \Gamma(q+1)
\end{align*}
where $\Gamma$ denotes the $\Gamma$-function -- see Lemma \ref{lemma integral estimates} for details. Additionally, if $i=1, j=2$ or $i=j=2$, then
\begin{align}\notag\label{eq w1q estimate}
\int_\Omega \bigg|\frac{x_ix_j}{\|x\|^2+\eps^2} &- \frac{x_ix_j}{\|x\|^2}\bigg|^q\dx \leq \pi \int_0^1r^{2q}\left|\frac1{r^2}- \frac1{r^2+\eps^2}\right|^q\,r\dr
	= \pi \int_0^1r^{2q}\left(\frac{\eps^2}{r^2(r^2+\eps^2)}\right)^qr\dr\\
	 &= \pi\,\eps^{2q}\int_0^1 (r^2+\eps^2)^{-q}r\dr = \frac{\pi\,\eps^{2q}}{2(1-q)} \int_0^1\frac{d}{dr} (r^2+\eps^2)^{1-q}\dr \leq  \frac{\pi \,\eps^{2q+2-2q}}{2(q-1)},
\end{align}
We combine the two estimates to obtain
\[
\|\nabla u_\eps - \nabla u\|_{L^q} \leq C\left\{\left(\frac1{q-1} + \Gamma(q+1)\right)^{1/q} \right\}\,\eps^{2/q}.
\]
By Stirling's approximation, we see that $\Gamma(q+1)^{1/q} = q/e + o(q)$ as $q\to\infty$, justifying the form of the estimate given in the theorem statement. For $q=1$, there is an additional logarithmic blow-up in the first integral, which for $q>1$ leads to the $(q-1)^{-1}$ contribution. 

{\bf Step 2.4: Gradient convergence on the boundary.} On the boundary, we have
\[
\nabla u_\eps -\nabla u^* = \left( 0, \frac{\log\big(\|x\|^2\big) - \log\big(\|x\|^2+\eps^2\big)}2\right),
\]
so
\begin{align*}
\int_{B_1(0)\cap \partial \R^2_+} \|\nabla u_\eps - \nabla u^*\|^q\dx_1 &= 2^{1-q} \int_0^1\log^q\left(1 + \frac{\eps^2}{r^2}\right)\dr.
\end{align*}
The expression strongly resembles the integral estimated in Step 2.3, but it lacks the factor $r$ of the polar coordinate functional determinant in two dimensions. We see that
\begin{align*}
\int_{B_1(0)\cap \partial \R^2_+} \|\nabla u_\eps - \nabla u^*\|^q\dx_1 &=  2^{1-q} \int_0^1\log^q\left(1 + \frac{\eps^2}{r^2}\right)\dr\\
	&= 2^{1-q}\eps \int_0^1\log^q\left(1 + \left(\frac{\eps}{r}\right)^2\right) \left(\frac r\eps\right)^2 \frac{\eps}{r^2}\dr\\
	&=2^{1-q}\eps \int_\eps^\infty \frac{\log^q(1+z^2)}{z^2}\dz\\
%	&\leq 2^{-q}\eps \left(\int_\eps^1 z^{2q-2}\dz + \int_1^\infty \frac{\log^q(1+z^2)}{z^3}(2z)\dz \right)\\
%	&\leq 2^{-q}\eps \left(\frac1{2q-1} + \int_0^\infty t^q\,e^{-3/2 t+t}\dt\right)\\
%	&\leq \int_0^2 \frac{\log^q(1+z^2)}{z^3}(2z)\dz\\
%	&\leq 2^{-q}\eps \,\int_0^\infty t^q\,e^{-3/2 t+t}\dt\\
	&= 2^{-q}\left( 1+ 2^{q+1}\,\Gamma(q+1)\,\right)\eps
\end{align*}
as in Appendix \ref{appendix integrals}.

{\bf Step 2.5: $W^{2,p}$-estimate.} We study  
\begin{align*}
\pi \partial_{x_1}\partial_{x_1}(u^*-u_\eps) &= 2\left(\frac{x_1^2x_2}{(\|x\|^2+\eps^2)^2}-\frac{x_1^2x_2}{\|x\|^4}\right) + \frac{x_2}{\|x\|^2+\eps^2}-  \frac{x_2}{\|x\|^2}.
\end{align*}
as a representative example for the second derivatives. All other derivatives can be studied by similar means. Here
\begin{align*}
\int_\Omega \left| \frac{x_1^2x_2}{(\|x\|^2+\eps^2)^2}-\frac{x_1^2x_2}{\|x\|^4} \right|^p
	&\dx
	\leq \int_\Omega \|x\|^{3p}\left|\frac{\|x\|^4 - (\|x\|^2+\eps^2)^2}{\|x\|^4(\|x\|^2+\eps^2)^2}\right|^p\dx\\
	&= \frac\pi2 \int_0^1 r^{3p}\left|\frac{\eps^4 + 2 r^2\eps^2}{r^4(r^2+\eps^2)^2}\right|^pr\dr\\
	&\leq 2^p\pi \int_0^1\left(\frac{\eps^4}{r(r^2+\eps^2)^2}\right)^pr + 2^p \left(\frac{\eps^2r}{(r^2+\eps^2)^2}\right)^pr\dr\\
	&\leq 4^p\pi \int_0^\eps r^{1-p} + \eps^{-2p}r^{p+1}\dr +4^p\pi \int_\eps^1\eps^{4p} r^{1-5p}+ \eps^{2p}r^{1-3p}\dr\\
	&\leq 4^p\pi\left(\frac1{2-p}\eps^{2-p}+ \frac1{p+2}\eps^{-2p+p+2} + \frac1{5p-2}\eps^{4p+2-5p} + \frac1{3p-2}\eps^{2p+2-3p}\right)\\
	&\leq 4^p\frac{C}{2-p}\,\eps^{2-p}\\
\int_\Omega \left|\frac{x_2}{\|x\|^2+\eps^2}- \frac{x_2}{\|x\|^2}\right|^p
	&\dx 
	\leq \frac\pi2\int_0^1 \left(\frac{\eps^2}{r(r^2+\eps^2)}\right)^pr\dr \leq 2^p\pi \left(\int_0^\eps r^{1-p}\dr + \int_\eps^1 \eps^{2p}\,r^{1-3p}\dr\right)\\
 & =  \frac{2^p\,C}{2-p}\eps^{2-p}.
\end{align*}
Overall, we see that for $p<2$
\[
\left(\int_\Omega \|D^2u_\eps - D^2u^*\|^p\dx\right)^\frac1p \leq \frac{C}{(2-p)^{1/p}}\,\eps^\frac{2-p}p.\qedhere
\]
\end{proof}

A few observations are in order.

\begin{remark}[Convergence in $W^{1,\infty}$ and $W^{1,1}$.]
The constants in the $W^{1,q}$ convergence estimate blow up both as $q\to 1$ and $q\to\infty$. Since $u^*\notin W^{1,\infty}$, we cannot expect the gradients to remain bounded in $L^\infty$, meaning that the constants {\em must} blow up. For $q=1$, the convergence is slower than naively expected by a logarithmic factor, as can easily be seen by a modification of Step 2.3:
\[
\eps^2 \int_0^1 \frac{r}{r^2+\eps^2}\dr \leq \eps^2\left(\int_0^\eps \frac{r}{\eps^2}\dr + \int_\eps^1 \frac1r\dr\right) = \eps^2\left(\frac12+|\log\eps|\right).
\]
\end{remark}

\begin{remark}[Higher dimensions]
The same construction can be applied in half-spaces of arbitrary dimension. Given a half-space $H= \{x\in\R^d : x_d>0\}$ and a boundary condition $u(x_1, \dots, x_{d-1},0) = \sigma(w^Tx+b)$, we may simply choose an approximation $\hat u_\eps(x) = \|\hat w\| u_\eps(\hat w/\|\hat w\|\cdot x + b/\|\hat w\|, x_d)$ where $\hat w = (w_1, \dots, w_{d-1},0)$. 
\end{remark}

Equation \eqref{eq laplacian ueps} shows that $\Delta u_\eps$ contains a term of the form $x_2/\|x\|$, so $u_\eps = \bar u+ v_\eps \notin H^2(\Omega)$ (but $v_\eps\in H^2(\Omega)$ as seen above).
In the following, we explore an alternative approximation of $u^*$ which exactly satisfies the Laplace equation, but only approximately matches the boundary values.

\begin{lemma}\label{lemma h2 approximation}
 The function $\tilde u_\eps(x) = u^*(x_1, x_2+\eps)$ satisfies the norm growth bounds
\[
\|D^2 \tilde u_\eps \|_{L^2(\Omega)} \lesssim \sqrt{|\log\eps|}, \qquad \|\tilde u_\eps\|_{\B(\Omega)} \lesssim |\log\eps|
\]
and $\Delta \tilde u_\eps \equiv 0$. Furthermore, the closeness estimates
\begin{align*}
\|\tilde u_\eps-u^*\|_{L^q(B_1(0)\cap \partial \R^2_+)} &\lesssim q\eps &&\forall\ q \in [1,\infty)\\
\|\tilde u_\eps - u^*\|_{L^\infty(B_1(0)\cap \partial \R^2_+)} &\lesssim \eps|\log\eps|\\
\|\nabla (\tilde u_\eps - u^*)\|_{L^1(B_1(0)\cap \partial \R^2_+)} &\lesssim \eps|\log\eps|\\
\|\nabla (\tilde u_\eps - u^*)\|_{L^q(B_1(0)\cap \partial \R^2_+)} &\lesssim \left(q + \frac1{q-1}\right)\eps^{1/q}  &&\forall\ q \in (1,\infty)
\end{align*}
on the boundary and
\begin{align*}
\|\tilde u_\eps - u^*\|_{L^q(\Omega)} &\lesssim q\eps\\
\|\tilde u_\eps - u^*\|_{L^\infty(\Omega)} &\lesssim \eps|\log\eps|\\
\|\nabla (\tilde u_\eps - u^*)\|_{L^q(\Omega)} &\lesssim \frac1{2-q}\,\eps && \forall\ q\in [1,2)\\
\|\nabla (\tilde u_\eps - u^*)\|_{L^2(\Omega)} &\lesssim \eps\,\sqrt{|\log\eps|} \\
\|\nabla (\tilde u_\eps - u^*)\|_{L^q(\Omega)} &\lesssim \left(q + \frac1{q-2}\right)\eps^{2/q} &&\forall\ q\in(2,\infty)\\
\|D^2(\tilde u_\eps - u^*)\|_{L^p(\Omega)} & \lesssim(2-p)^{-1/p}\,\eps^\frac{2-p}p && \forall\ p \in [1,2)
\end{align*}
in the interior hold for $\tilde u_\eps$.
\end{lemma}

\begin{proof}
{\bf Step 1: Norm bounds.} Since $x_1\left(\frac1\pi \,\arctan(x_1/x_2) +\frac12\right)$ is a Barron function, we focus on the logarithmic corrector term
\[
x_2\, \log(\|x\|^2) = x_2\,\log\left(x_2^2 \left(1+ \frac{x_1^2}{x_2^2}\right)\right) = 2x_2\log(x_2) + x_2\,\log\left(1+ \left(\frac{x_1}{x_2}\right)^2\right),
\]
which can be analyzed by mostly one-dimensional tools, imitating the first step of the proof of Lemma \ref{lemma barron regularity}. If $x_2\geq \eps$ and $|x_1|\leq 1$, we note that $|x_1/x_2|\leq 1/\eps$. The functions
\[
g_1(x) =  2\,x\log x, \qquad g_2(z) = \log\left(1+z^2\right)
\]
satisfy
\begin{align*}
\int_\eps^1 |g_1''(x)|\sqrt{1+x^2}\dx &\leq 4\int_\eps^1 \frac1x \dx = 2\,|\log\eps|,\\
\int_{-1/\eps}^{1/\eps} |g_2''(z)|\,\sqrt{1+z^2}\dz &=2 \int_{-1/\eps}^{1/\eps} \frac{|1-z^2|}{(1+z^2)^{3/2}}\dz \sim |\log\eps|
\end{align*}
since the integrand decays like $1/z$ for large $z$. Consequently $\|\tilde u_\eps\|_{\B(\Omega)}\lesssim |\log \eps|$, i.e.\ there exists a Barron function of Barron norm $\sim|\log \eps|$ which coincides with $\tilde u_\eps$ in $\Omega$. 

As in the proof of the first claim of Lemma \ref{lemma sobolev regularity}, we find that 
\[
\|D^2\tilde u_\eps\|_{L^2(\Omega)}^2 \lesssim \int_\eps^1 \left(\frac r{r^2}\right)^2 r\dr = \int_\eps^1 \frac1r \dr = |\log\eps|.
\] 

{\bf Step 2a: $L^\infty$- and $L^q$-closeness in the interior.} Since 
\[
\bar u (x) = x_1\left(\frac1\pi \,\arctan\left(\frac{x_1}{x_2}\right) + \frac12\right)
\]
is a Barron function, it is in particular Lipschitz continuous and $\big|\bar u(x) - \bar u(x_1, x_2+\eps)\big|\lesssim \eps$. The logarithmic term satisfies
\begin{align*}
\big|x_2\log(\|x\|^2) - (x_2+\eps)\,\log\big(x_1^2+(x_2+\eps)^2\big)\big| & \leq x_2 \left|\log\left(\frac{x_1^2+x_2^2}{x_1^2+(x_2+\eps)^2}\right)\right| + \eps \,\big|\log\big(x_1^2+(x_2+\eps)^2\big)\big|\\
	&\leq x_2\,\log\left(1+  \frac{2\eps x_2+\eps^2}{x_1^2+x_2^2}\right) + 2\eps|\log\eps|
\end{align*}
since the second term is largest for $x_1=x_2=0$ as $\|x\|$ cannot become large. The first term is maximal when $x_1=0$, i.e.\ we must consider the one-dimensional function
\[
x_2\,\log\left(1+ 2\frac{\eps}{x_2} + \frac{\eps^2}{x_2^2}\right) = x_2\log\left(\left(1+ \frac\eps{x_2}\right)^2\right) = 2x_2 \,\log\left(1+ \frac{\eps}{x_2}\right).
\]
This expression is largest when 
\[
0= \frac{d}{dx_2} \big(x_2\,\log(1+\eps/x_2)\big) = \log\left(1+\frac\eps{x_2}\right) + \frac{x_2}{1+ \frac\eps{x_2}}\,\left(-\frac\eps{x_2^2}\right) = \log\left(1+ \frac\eps{x_2}\right) - \frac{\eps}{x_2+\eps}
\]
so
\[
2x_2 \,\log\left(1+ \frac{\eps}{x_2}\right) \leq \frac{2x_2\eps}{\eps+x_2} \leq 2\eps\frac{x_2}{x_2+\eps} \leq 2\eps.
\]
Putting all terms together, we find that $\|\tilde u_\eps - u^*\|_{L^\infty(\Omega)}\leq 2\eps|\log\eps| + C\eps$ with $C$ depending also on the Lipschitz-constant of the arctangent term. On the other hand, if $q<\infty$ then we use $x_2\geq 0$ to analyze the logarithmic term:
\begin{align*}
\|u^*-\tilde u_\eps\|_{L^q(\Omega)} &\leq C\eps \,|\Omega|^{1/q}  +  \left(\int_\Omega \big|\eps\,\log\big(x_1^2 + x_2^2 + \eps^2\big)\big|^q\dx\right)^{1/q}
	\leq  C\eps \,|\Omega|^{1/q} + \eps\left(\int_0^1 |\log(r^2)|^q\,r\dr\right)^{1/q}\\
	&=  C\eps \,|\Omega|^{1/q} + 2\eps\left(\int_0^1 |\log(r)|^q\,r\dr\right)^{1/q}
	\lesssim 4\eps \,|\Omega|^{1/q} + q\eps
\end{align*}
as in Step 2.3 of the proof of Lemma \ref{lemma barron regularity}.

{\bf Step 2b:  $L^\infty$- and $L^q$-closeness on the boundary.} $L^\infty$-closeness follows just as above. On the boundary, the logarithmic term simplifies to
\[
\big|x_2\log(\|x\|^2) - (x_2+\eps)\,\log\big(x_1^2+(x_2+\eps)^2\big)\big|  = \eps\,|\log(x_1^2+\eps^2)|.
\]
Similar to the proof of Lemma \ref{lemma barron regularity}, the logarithm is in $L^q(-1,1)$ for all $q<\infty$ and
\[
\int_0^1|\log x|^q \dx = \Gamma (q+1) \qquad \Ra \quad \|u^*-\tilde u\|_{L^q(B_1(0) \cap \partial \R^2_+)} \leq q\,\eps \approx\dots
\]

{\bf Step 3a: First derivatives in the interior.} Using the derivatives computed in Lemma \ref{lemma sobolev regularity} and the properties of logarithms as in Step 2, we find that
\begin{align*}
\big(\partial_{x_1}u^* - \partial_{x_1}\tilde u_\eps \big)(x)&= \frac1\pi \left(\arctan\left(\frac{x_1}{x_2}\right) - \arctan\left(\frac{x_1}{x_2+\eps}\right)\right) 
	\\
\big(\partial_{x_2}u^* - \partial_{x_2}\tilde u_\eps\big)(x) 
%&= -\frac 1{2\pi}\log(\|x\|^2) + \frac1{2\pi} \log\big(x_1^2+(x_2+\eps)^2\big)\\
%	& = \frac1{2\pi}\,\log\left(\frac{x_1^2+(x_2+\eps)^2}{x_1^2+x_2^2}\right)\\
	&= \frac1{2\pi} \,\log\left(1+ \frac{\eps x_2 + \eps^2}{x_1^2 + x_2^2}\right).
\end{align*}
Note that
\begin{align*}
\left|\arctan\left(\frac {x_1}{x_2}\right) - \arctan\left(\frac{x_1}{x_2+\eps}\right)\right| &= \int_{\frac{|x_1|}{x_2+\eps}}^{|x_1|/x_2} \frac 1{1+s^2}\ds\\
	&= |x_1| \int_{|x_1|/(x_2+\eps)}^{|x_1|/x_2} \frac{1}{(|x_1|/s)^2 + |x_1|^2}\,\frac{|x_1|}{s^2}\ds\\
	&= \int_{x_2}^{x_2+\eps} \frac{|x_1|}{t^2+x_1^2}\dt \leq \frac{\eps|x_1|}{x_2^2+x_1^2}
\end{align*}
and, using that $\arctan$ is an odd function with $0\leq \arctan(z) \leq \pi/2$ for $z\geq 0$, we find that
\[
\left|\arctan\left(\frac {x_1}{x_2}\right) - \arctan\left(\frac{x_1}{x_2+\eps}\right)\right| \leq \min\left\{\frac\pi2, \, \frac{\eps|x_1|}{x_1^2+x_2^2}\right\}.
\]
In particular, we conclude that
\begin{align*}
\|\partial_{x_1}(u^*-\tilde u_\eps)\|_{L^q(\Omega)}^q &\leq \pi\int_0^1 \min\left\{\frac\pi2, \frac{\eps r}{r^2}\right\}^q r\dr\\
	&= \pi\int_0^{2\eps/\pi}\left(\frac\pi 2\right)^q r\dr + \int_{2\eps/\pi}^1 \eps^q r^{1-q}\dr\\
	&= \,\left(\frac\pi 2\right)^{q-1}\eps^2 + \eps^q \begin{cases} \frac1{2-q} &q<2\\ |\log(2\eps/\pi)| &q=2\\ \frac1{q-2}\left(\frac{2\eps}\pi\right)^{2-q} &q>2\end{cases}.
\end{align*}
As in Step 2.3 in the proof of Lemma \ref{lemma barron regularity}, we have 
\[
\|\partial_{x_2}(u^*-\tilde u_\eps)\|_{L^q(\Omega)} \lesssim \begin{cases}\left(\frac1{q-1}+q\right) \eps^{2/q}&q>1\\ \eps^2|\log\eps| & q=1.\end{cases}
\]
 for $q>1$. Overall, we find that
 \[
 \|\nabla (\tilde u_\eps-u^*)\|_{L^q(\Omega)} \lesssim \begin{cases} \eps &q<2\\ \eps\sqrt{|\log\eps|} & q=2\\ \eps^{2/q} & q>2.\end{cases}
 \]

{\bf Step 3b: First derivatives on the boundary.} 
On the boundary, we find that
\begin{align*}
\|\partial_{x_1}u^*-\partial_{x_1}\tilde u_\eps\|_{L^q(B_1 \cap \partial\R^2_+)} 
	&\leq \frac1\pi \left(2\int_0^1 \left|\frac\pi2-\arctan\left(\frac{x_1}\eps\right)\right|^q\dx_1\right)^{\frac1q}\\
	&\leq \frac{2^{1/q}}\pi \left(\int_0^{2\eps/\pi} \left(\frac\pi2\right)^q\ds + \int_{2\eps/\pi}^1 \left(\frac {\eps}s\right)^q\ds\right)^{\frac1q}\\
	&\leq \frac{2^{1/q}}\pi \left(\left(\frac\pi2\right)^q\,\frac{2\eps}\pi + \frac{\eps^q}{q-1} \left(\frac{2\eps}\pi\right)^{1-q} \right)^{\frac1q}\lesssim \frac{\eps^{1/q}}{q-1}\\
\|\partial_{x_2}(\tilde u_\eps - u^*)\|_{L^q(B_1\cap \partial \R^2_+)} &\leq \frac1{2\pi} \left(\int_{-1}^1 \log^q\left(1+ \frac{\eps ^2}{x_1^2}\right)\dx_1\right)^\frac1q \lesssim q\,\eps^{1/q}
\end{align*}
as in Step 2.4 of the proof of Lemma \ref{lemma barron regularity}. For $q=1$, the integral of $1/s$ becomes a logarithm instead.

{\bf Step 4: Second derivatives.} We only consider the closeness of second derivatives in the interior since the tangential second derivative of $\tilde u_\eps$ on the boundary is a smooth function by design while that of $u^*$ is a Dirac measure at the origin so that the two {\em cannot} be close  in any integral norm.

Since $\partial_{x_1}\partial_{x_1}u^* = -\partial_{x_2}\partial_{x_2}u^*$ and the same for $\tilde u_\eps$, we only need to consider one of the derivatives. Again, we use the expressions computed in Lemma \ref{lemma sobolev regularity} to estimate
\begin{align*}
\big(\partial_{x_1}^2u^* - \partial_{x_1}^2\tilde u_\eps \big)(x)
	&= \frac{x_2}{\pi \,\|x\|^2} - \frac{x_2+\eps}{\pi\,\|(x_1, x_2+\eps)\|^2} &&= \frac{x_2}\pi\left(\frac1{\|x\|^2} - \frac{1}{\|x+\eps e_2\|^2}\right) + \frac{\eps}{x_1^2+(x_2+\eps)^2}\\
\big(\partial_{x_1}\partial_{x_2} u^* - \partial_{x_1}\partial_{x_2}\tilde u_\eps\big)(x)
	& = \frac{x_1}{\pi \,\|x\|^2} - \frac{x_1}{\pi\,\|(x_1, x_2+\eps)\|^2} &&= \frac{x_1}\pi\left(\frac1{\|x\|^2} - \frac{1}{\|x+\eps e_2\|^2}\right)
\end{align*}
The first term in both expressions has been studied in Step 2.5 of the proof of Lemma \ref{lemma barron regularity} while the final term is easy to analyze:
\[
\int_\Omega \left|\frac{\eps}{x_1^2+(x_2+\eps)^2}\right|^p\dx\leq \int_0^1 \left(\frac{r}{r^2+\eps^2}\right)^pr\dr \leq \int_0^\eps \frac{r^{p+1}}{\eps^{2p}}\dr + \int_\eps^1 r^{1-p}\dr \leq \frac1{p+2}\eps^{2-p} + \frac{\eps^{2-p}}{2-p}. \qedhere
\]
\end{proof}

\subsection{Quadrants}
\label{section quadrants}

\begin{figure}
\includegraphics[width = .47\textwidth]{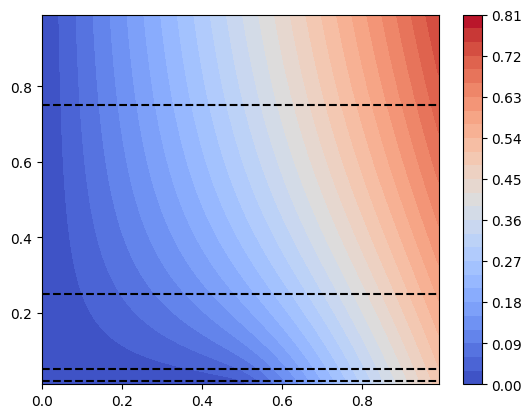}\hfill
\includegraphics[width = .49\textwidth]{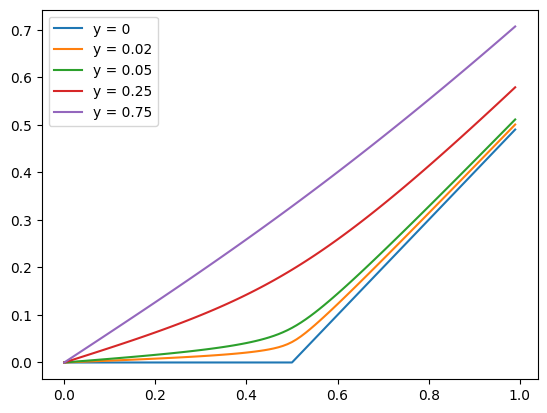}
\caption{\label{figure quadrant 1}
Harmonic function on the sector $\{x_1>0, \,x_2>0\}$ with boundary condition $\sigma(x_1+w_1x_2 - 0.5)$ with $w_2\leq 0$, created from $u^*$ by the method of mirror charges as outlined in Section \ref{section quadrants}. Note that Figures \ref{figure harmonic function}, \ref{figure quadrant 1} and \ref{figure quadrant 2} have different color scales.
}
\end{figure}

%\begin{figure}\label{figure quadrant 2}
%\includegraphics[width = .47\textwidth]{./sector_contour_2.png}\hfill
%\includegraphics[width = .49\textwidth]{./sector_profiles_2.png}
%\caption{
%Harmonic function on the sector $\{x_1>0, \,x_2>0\}$ with boundary condition $\sigma(x_1+4x_2 - 0.8)$, created by the method of mirror charges as outlined in Section \ref{section quadrants}.
%}
%\end{figure}

\begin{figure}
\includegraphics[width = .47\textwidth]{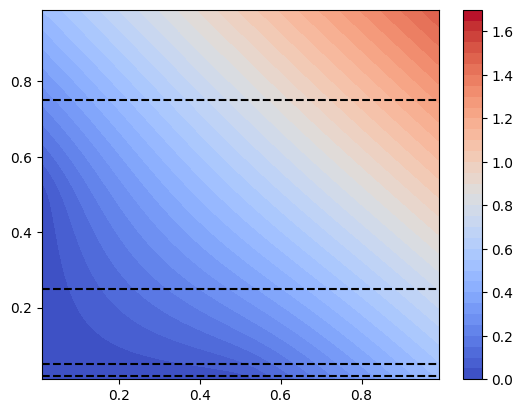}\hfill
\includegraphics[width = .49\textwidth]{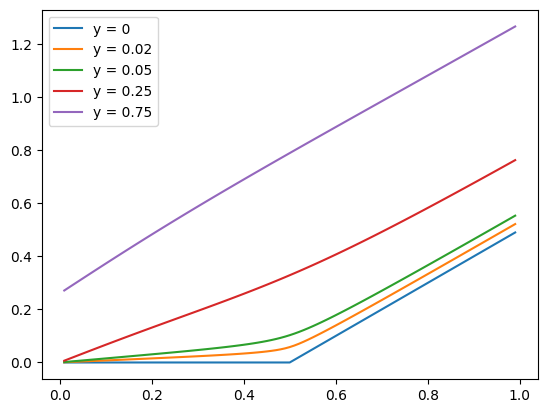}
\caption{\label{figure quadrant 2}
Harmonic function on the sector $\{x_1>0, \,x_2>0\}$ with boundary condition $\sigma(x_1+x_2 - 0.5)$, created by the method of mirror charges as outlined in Section \ref{section quadrants}.
}
\end{figure}

Building on the analysis on the half-space, we consider a quadrant $Q:= \{x\in \R^2: x_1, x_2>0\}$. Here we solve the PDE
\[
\begin{pde}
\Delta u &= 0 &\text{in }Q\\
u &= \sigma(w\cdot x+b) &\text{on }\partial Q.
\end{pde}
\]
As on the half-space, solutions are generally not unique as we can add e.g.\ a multiple of the harmonic function $h(x_1,x_2) = x_1x_2$ which vanishes on the boundary. However, they are `more unique' in the sense that two solutions which differ by a function which grows subquadratically at infinity (e.g.\ a uniformly continuous function) automatically coincide: Assume that 
\[
\Delta u_1 = \Delta u_2 \text{ in }Q, \quad u_1 = u_2 \text{ on }\partial Q, \quad \limsup_{|X|\to \infty, \:x\in Q} \frac{|(u_1-u_2)(X)|}{|X|^2} = 0.
\]
The function $h:= u_1 - u_2$ is harmonic in the quadrant $Q$ and vanishes on its boundary. We can extend $h$ to a harmonic function on the entire half-space $\R^2_+$ by odd reflection in $x_1$ as $h(-x_1, x_2) = -h(x_1,x_2)$ and note that the extended function still satisfies the sub-quadratic growth condition. By Liouville's Theorem in the version \cite[Theorem 1.1]{vaishampayan2024solving}, we find that $h(x_1, x_2) = \lambda x_2$ for some $\lambda \in\R$. Repeating the argument with the roles of $x_1, x_2$ exchanged, we deduce that $h(x_1, x_2) = \mu x_1$ for some $\mu\in\R$, which leads us to conclude that $\mu = \lambda = 0$ and thus $u_1 = u_2$. In particular, solutions that can be approximated well by ReLU networks on large domains automatically are unique.

We distinguish three situations for the approximation of the unique solution:

\begin{enumerate}
\item $\sigma(w\cdot x+b)$ is affine linear on $\partial Q$. Then either $u\equiv0$ or $u(x) = w\cdot x+b$ is a linear (in particular, Barron) harmonic function with the correct boundary values.

\item The line $\{x : w\cdot x+b = 0\}$ intersects $Q$ and crosses $\partial Q$ once.

\item The line $\{x : w\cdot x+b = 0\}$ intersects $Q$ and crosses $\partial Q$ twice.
\end{enumerate}

Let us consider the non-trivial second and third cases in detail. Since $w\cdot x+b = \sigma(w\cdot x+b) - \sigma(-w\cdot x - b)$ and linear functions are both Barron and harmonic, it suffices for us to consider the case that $b<0$. A crossing occurs at $(x,0)$ if $w_1x +b =0$ and at $(0,y)$ if $w_2y+b=0$. If $b<0$, then the line $\{x: w^Tx+b\}$ intersects the positive $x$-axis $\{(x,0) : x>0\}$ if and only if $w_1>0$ and the positive $y$-axis if and only if $w_2>0$.

Consider a single crossing first. We construct $u$ from $u^*$ inspired by the method of ``mirror charges'' in the computation of Green's functions.
Without loss of generality, we may assume that $w_1>0$ and $w_2\leq 0$. Define
\begin{equation}
u(x_1, x_2) = w_1\,u^*(x_1+b/w_1, x_2) - w_1u^*(-x_1+b/w_1, x_2).
\end{equation}
Then by construction $u$ is harmonic in the half-plane $x_2>0$ -- in particular the quadrant $Q$ -- and 
\begin{align*}
u(x_1, 0) &= w_1\,\sigma\left(x_1+b/w_1\right) - w_1\,\sigma\left(-x_1+\frac b{w_1}\right) = \sigma(w_1x_1+w_2\cdot0+b) - 0 &\qquad \forall\ x_1\geq0\\
u(0,x_2) &= w_1\left(u^*\left(\frac b{w_1}, x_2\right) - u^*\left(\frac b{w_1}, x_2\right) \right) = 0  = \sigma(w_1\cdot 0 + w_2x_2+b)& \forall\ x_2\geq 0
\end{align*}
on $\partial Q$. In this situation, the exact value of $w_2\leq 0$ is irrelevant (Figure \ref{figure quadrant 1}). Similarly, if $b<0$ and there are two crossings (i.e.\ $w_1, w_2>0$), we can explicitly construct
\begin{align*}
u(x_1,x_2) &= w_1\,u^*\left(x_1+ \frac b{w_1}, \,x_2\right) - w_1\,u^*\left(-x_1+\frac b{w_1}, \,x_2\right)\\
	&\qquad\quad + w_2\,u^*\left(x_2+ \frac b{w_2}, \,x_1\right) - w_2\, u^*\left(-x_2+\frac b{w_2}, \,x_1\right),
\end{align*}
see Figure \ref{figure quadrant 2}.
The regularity/approximation results in Lemmas \ref{lemma barron regularity} and \ref{lemma h2 approximation} can therefore be used to obtain approximation rates on bounded sets also locally in the quadrant $Q$. In principle, if one of the axis intersections is much further from the origin than the length-scale of interest, a different analysis might be in order to zoom in on the behavior far from the singularity. Here, the function $u^*$ can indeed be represented as a Barron function due to its smoothness. This is easier than the analysis given in Lemma \ref{lemma barron regularity}, but we do not provide details since our main focus will be rectangular domains, where the length-scale of interest is finite and well defined.

Unlike the situation of a half-space, the quadrant analysis does not generalize immediately to higher dimensions: While the intersection of two lines is a single point, the intersection of the plane $w\cdot x+b = 0$ with the boundary of the quadrant $x_1, x_2>0$ in three dimensions is most generally a piecewise linear curve with two segments. The dimension cannot be reduced as readily here.
%When working on a half-space $H= \{x_1>0\}$, the $d$ parameters $w_1, \dots, w_{d-1},b$ suffice to fully specify the boundary condition $\sigma(w\cdot x+b)$ since $x_d =0$ on $\partial \R^2_+$.
We present example plots of both non-trivial situations and mirror charge solutions in Figures \ref{figure quadrant 1} and \ref{figure quadrant 2}.

\subsection{Rectangular domains}

\begin{figure}
\includegraphics[width=.48\textwidth]{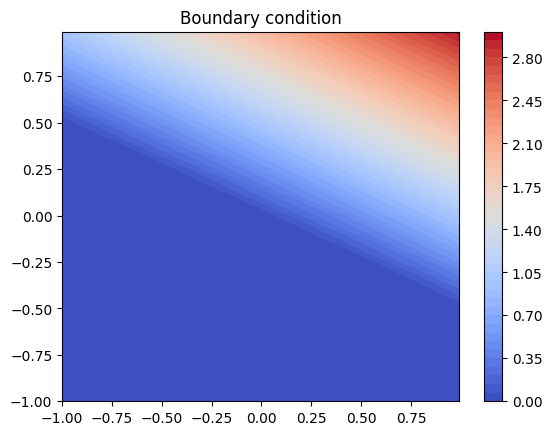}\hfill
\includegraphics[width=.48\textwidth]{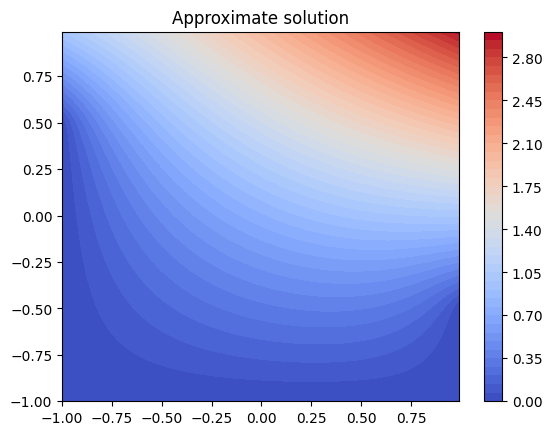}
\caption{{\bf Left:} Single neuron boundary condition. {\bf Right:} Solution of the Laplace equation on the square with the boundary condition on the left, approximated by a finite difference method.}
\end{figure}
Finally, we are ready to conclude the proof of Theorem \ref{theorem main}. We first establish the result in the case of single neuron boundary values and harmonic functions.

\begin{lemma}\label{lemma rectangle}
Let $R:= [a,b]\times [c,d]$ be a rectangular domain in $\R^2$. There exists a unique weak solution $u_{(w,b)} \in H^1(R)$ of the equation 
\[
\begin{pde}
-\Delta u&= 0 & \text{in }R\\
u &= \sigma(w\cdot x+b)& \text{on }\partial R
\end{pde}
\]
for any given $w\in\R^2$ and $b\in\R$ and it satisfies the following:
\begin{enumerate}
\item $u_{(w,b)} \in W^{1,q}(R)\cap W^{2,p}(R)$ and
\[
\|u_{(w,b)}\|_{W^{1,q}(R)} \lesssim q\sqrt{\|w\|_2^2 + b^2}, \qquad   \|u_{(w,b)}\|_{W^{2,p}(R)} \lesssim \left(\frac1{2-p}\right)^\frac1p \sqrt{\|w\|_2^2 + b^2}
\]
for all $q\in(1,\infty)$ and $p\in[1,2)$.

\item For every $\eps\in(0,1/2)$, there exists $u_{(w,b;\eps)}\in \B(\R^2)$ and a universal constant $C>0$ such that
\begin{align*}
\|u_{(w,b;\eps)}\|_\B &\leq C\,\sqrt{\|w\|_2^2 + b^2}\,|\log\eps|\\
\|u-u_{(w,b;\eps)}\|_{L^\infty(R)} &\leq C\sqrt{\|w\|_2^2 + b^2}\,\eps\\
\|u-u_{(w,b;\eps)}\|_{W^{1,q}(R)} &\leq C\left(\frac1{q-1} + q\right) \sqrt{\|w\|_2^2 + b^2}\,\eps^{2/q}&& \forall\ q\in (1,\infty)\\
\|u-u_{(w,b;\eps)}\|_{W^{2,p}(R)} &\leq C \left(\frac1{2-p}\right)^{1/p} \sqrt{\|w\|_2^2+b^2}\,\eps^{(2-p)/p}&& \forall\ p\in [1,2).
\end{align*}
\end{enumerate}
\end{lemma}

\begin{proof}
The boundary condition $g(x) = \sigma(w\cdot x+b)$ satisfies $g\in \B \subseteq W^{1,\infty}(R) \subseteq H^{1}(R)$, so the existence of a unique weak solution $u^\sharp$ follows from \cite[Theorem 8.3]{gilbarg2015elliptic}. 

If the line $\{x : w^Tx+b = 0\}$ intersects the open rectangle $R$, then it intersects the boundary $\partial R$ in two distinct points $X^1, X^2\in\R^2$. Without loss of generality, we assume that $X^1$ lies on the line $\{x_2=0\}$ and that the nearest corner is the origin $(0,0)$. We can select $u_{w,b,1}$ such that $\Delta u_{w,b,1} = 0$ in the quadrant $\{x_1, x_2 >0\}$ and such that $u_{w,b,1} = \sigma(w^Tx+b)$ on the ray $\{x_1>0, x_2=0\}$ and $u_{w,b,1}$ is linear on the ray $\{x_1=0, x_2>0\}$. 

Analogously, we select $u_{w,b,2}$ to match the boundary condition around $X^2$ on the side containing $X^2$ and making $u_{w,b,2}$ to be linear on the nearest orthogonal side to $X^2$ (with an arbitrary tie-break mechanism in the case of a tie). Then $\widehat u = u_{w,b,1} + u_{w,b,2}$ satisfies $\Delta \widehat u = 0$ inside the rectangular region $R$ and $\widehat u - \sigma(w^Tx+b)$ is an infinitely smooth function on $\partial R$ satisfying estimates on all (tangential) derivatives {\em independent of $\eps$} since $u^*$ is smooth except at the singularity, and so are $u_\eps, \tilde u_\eps$. The mirror charge technique ensures that $u_{w,b,i}$ is well-behaved both on the side containing $X_i$ and the nearest orthogonal side, `hiding' the singular derivatives except on the remaining two sides. The magnitude of tangential derivatives is thus governed by the aspect ratio of the rectangle $R$. 

Let $k\in \mathbb N_0$. By Lemma \ref{lemma extension from boundary}, we note that there exists $V\in C^k(R)$ such that
\[
 u^\sharp - \widehat u - V \in H^1_0(R), \qquad \Delta \big(u^\sharp - \widehat u - V\big) = - \Delta V.
\]
By positive one-homogeneity $\sigma(w^Tx+b) = \sqrt{\|w\|^2+b^2} \sigma(\hat w^Tx + \hat b)$ with $\|(\hat w, \hat b)\|_2 = 1$, we see easily that $\|V\|_{C^k(R)}$ scales linearly in $\|(w,b)\|_2$. 
By Lemma \ref{lemma regularity cube}, we conclude that $\| u^\sharp - \widehat u - V\|_{H^k(R)} \leq C\,\|(w,b)\|_2$ for a constant $C>0$ which depends on the rectangle $R$, but not $w,b$. Choosing $k=3$ and using that $H^3$ embeds into $W^{2,2},\ W^{1,\infty}$ and $\B$, we see that 
\[
u^\sharp = \widehat u + \underbrace{(u^\sharp - \widehat u - V)}_{\in \B \cap W^{1,\infty} \cap H^2} + \underbrace{V}_{\in \B \cap W^{1,\infty}\cap H^2}.
\]
Thus, both the regularity and approximation properties of $u^\sharp$ reduce to those of $\widehat u$ and thus ultimately to those of $u^*$. 
\end{proof}

We stated the Lemma for approximation by $u_\eps$ as in Lemma \ref{lemma barron regularity}, but analogous bounds to Lemma \ref{lemma h2 approximation} can be derived as well, constructing $\tilde u_{(w,b;\eps)}$. We have now collected all ingredients of our main result.

\begin{proof}[Proof of Theorem \ref{theorem main}]
If $g$ is a Barron function, then for every $m\in \mathbb N$, there exist functions $g_m(x) = \frac1m \sum_{i=1}^n a_{i,m}\,\sigma(w_{i,m}\cdot x+b_{i,m})$ such that 
\begin{itemize}
\item $\frac1m \sum_{i=1}^m |a_i| \sqrt{\|w_i\|^2 + |b_i|^2} \leq \|f\|_\B$ and
\item $\|g_m - g\|_{C^0(\partial R)} \lesssim \|f\|_\B \,m^{-1/2}$
\end{itemize}
by \cite[Theorem 12]{review_article}.
 For given $\eps>0$, we consider $u_m, u_{m,\eps}:R\to \R$ given by 
\[
u_{m}(x) = \frac1m\sum_{i=1}^m a_i\,u_{(w_{i,m}, b_{i,m})}, \qquad 
u_{m,\eps}(x) = \frac1m\sum_{i=1}^m a_i\,u_{(w_{i,m}, b_{i,m};\eps)}
\]
and note that $u_m$ and $u_{m,\eps}$ remain bounded in $W^{1,q}(R)$ for all $q<\infty$, in $W^{2,p}(R)$ for all $p<2$, and in the case of $u_{m,\eps}$, in Barron space with norm bounds as claimed in Theorem \ref{theorem main}. As $m\to \infty$, the estimates on $\|u_m - u_{m,\eps}\|$ follow from Lemma \ref{lemma rectangle}. Since all functions remain bounded, we can pass to the limit (weakly in the Sobolev spaces) and the uniform estimates survive the passage to the limit $m\to \infty$ due to the lower semi-continuity of the norm under weak convergence.

The limit of $u_m$ is the harmonic function with boundary values $g$ due to the continuity of the trace operator in the weak topology (use e.g.\ \cite[Theorem 3.10]{MR2759829}), the limit of $u_{m,\eps}$ the approximant. The fourth claim follows analogously, using Lemma \ref{lemma h2 approximation} in place of Lemma \ref{lemma barron regularity} to prove an analogue of Lemma \ref{lemma rectangle}.
\end{proof}

\section{Application: A priori error estimates for the Deep Ritz Method}\label{section deep ritz}

As an application of the regularity results, we obtain a priori error estimates for a regularized Deep Ritz method to solve the equation $-\Delta u = 0$ on a rectangular two-dimensional domain with Barron boundary values $g$. The Deep Ritz method builds on the characterization of the harmonic function as a minimizer of the Dirichlet energy
\[
E_{Dir}(u) = \frac12\int_\Omega \|\nabla u\|^2\dx 
\]
in the affine function class $\{u\in H^1(\Omega): u-g\in H_0^1(\Omega)\}$. Since $g\in \B\subseteq H^1(R)$, the problem is well-defined as there exist competitors of finite energy. For computational approximation purposes, we consider a discretized and regularized functional.

Let $R$ be a rectangular domain in $\R^2$, $\pi_R$ the uniform distribution on $R$ and $\pi_{\partial R}$ the uniform distribution on $\partial R$. Assume that $x_1, \dots, x_N$ are independent samples from $\pi_R$ and $z_1, \dots, z_n$ are independent samples from $\pi_{\partial R}$. Consider the empirical integral functionals
\[
E_{Dir, N}(u) = \frac {|R|}{2N}\sum_{i=1}^N \|\nabla u(x_i)\|^2, \qquad E_{bd, n}(u) = \frac{|\partial R|}{2n} \sum_{i=1}^n \big|(u-g)(z_i)\big|^2
\]
where $|R|,\ |\partial R|$ denote the area of $R$ and length of $\partial R$ respectively, their continuous counterparts
\[
E_{Dir}(u) = \frac12\int_R\|\nabla u\|^2\dx, \qquad E_{bd}(u) = \int_{\partial R}|u-g|^2\ds,
\]
and the regularized Deep Ritz loss functional
\[
\widehat E_{N,n}(a,W,b) = E_{Dir,N}(u_{(a,W,b)})+ \lambda\,E_{bd,n}(u_{(a,W,b)}) + \mu\big(\|a\|_2^2 + \|W\|_F^2 + \|b\|_2^2\big)
\]
with parameters $\mu,\lambda$. As above, $a, b$ are the vectors containing the variables $a_i, b_i$, $W$ is the $m\times 2$-matrix containing the vectors entries $w_i$ in its rows and
\[
u_{(a,W,b)}(x) = \sum_{i=1}^m a_i\,\sigma(w_i\cdot x+b_i).
\]
All norms are Euclidean ($\ell^2$-, Frobenius) norms. In the limit $m, N, n\to + \infty$, we recover an approximation to the regularized Deep Ritz functional
\[
E_{DR, \lambda,\mu}(u) = \frac12\int_R\|\nabla u\|^2 \dx + \frac\lambda2\,\int_{\partial R} \big|u-g|^2\ds + \mu \,\|u\|_\B
\]
where regularity and compliance with boundary values are merely enforced by penalty. Below, we also denote $E_{\lambda}(u) = E_{DR, \lambda,0}(u)$ with $\mu=0$ for brevity.

\begin{proof}[Proof of Theorem \ref{theorem deep ritz}]
{\bf Step 1. Energy and closeness.} Using that $|v-w|^2 = |v|^2 +2(w-v)\cdot w -|w|^2$ for all vectors $v,w$ in a Hilbert space and the divergence theorem, we observe that
\begin{align*}
\|\nabla (u-u^*)\|^2_{L^2(R)} &= \int_R \|\nabla u\|^2 - \|\nabla u^*\|^2 + 2 \,\nabla (u^*-u)\cdot \nabla u^*\dx\\
	&= \int_R \|\nabla u\|^2 - \|\nabla u^*\|^2\dx + 2\int_{\partial R}(u^*-u)\,\partial_\nu u^*\ds\\
	&\leq \int_R \|\nabla u\|^2 - \|\nabla u^*\|^2\dx + 2\|u-u^*\|_{L^2(\partial R)}\|\nabla u^*\|_{L^2(\partial R)}.
\end{align*}
So at least for the continuous limiting version of the energy, the `energy gap' controls the closeness between the Deep Ritz minimizer and the actual solution to the variational problem. We further deduce that if $E_{Dir}(u) \geq E_{Dir}(u^*)$, then
\begin{align*}
\|\nabla (u-u^*)\|^2_{L^2(R)} &\leq E_{Dir}(u) - E_{Dir}(u^*) + \frac{2\,\|\nabla u^*\|_{L^2(\partial R)}}{\sqrt\lambda} \,\sqrt{\lambda\,\|u-u^*\|_{L^2(\partial R)}^2}\\
	& \leq E_{\lambda}(u) - E_{Dir}(u^*) + 2\,\|\nabla u^*\|_{L^2(\partial R)} \,\sqrt{\frac{E_{\lambda}(u) - E_{Dir}(u^*)}\lambda}.
\end{align*}
Furthermore, note that
\[
||| u||| = \sqrt{ \|\nabla u\|_{L^2(\Omega)} + \|u\|_{L^2(\partial\Omega)}}
\]
is an equivalent norm to the standard $H^1$-norm, so moreover
\[
\|u\|_{H^1(\Omega)}^2 \lesssim |||u|||^2 \leq E_{\lambda}(u) - E_{Dir}(u^*) + 2\,\|\nabla u^*\|_{L^2(\partial R)} \,\sqrt{\frac{E_{\lambda}(u) - E_{Dir}(u^*)}\lambda} + \frac{E_\lambda(u) - E_{Dir}(u^*)}\lambda.
\]
{\bf Step 2. Generalization gap in Barron space.} Similar proofs in a more classical setting of data science can be found e.g.\ in \cite[Theorem E.1]{park2023minimum}. 
Since the unit ball of Barron space has finite Rademacher complexity \cite[Theorem 6]{weinan2019lei} and the `loss function' $u\mapsto (u-g)^2$ is Lipschitz-continuous with Lipschitz constant $2(A+\|g\|_\B) (1+ \max_{x\in R}\|x\|)$ on the set $\|u\|_\B\leq A$ and satisfies $|u|^2\leq (A+ \|g\|_\B)^2 \big(1+ \max_{x\in R} \|x\|\big)^2$ , with probability at least $\delta>0$ we have
\[
\sup_{\|u\|_\B\leq A} E_{bd}(u) \leq  E_{bd,n}(u) + \frac{ 4A\,(A+\|g\|_\B) (1+ \max_{x\in R}\|x\|)}{\sqrt n}+ (A+ \|g\|_\B)^2 \big(1+ \max_{x\in R} \|x\|\big)^2\sqrt{\frac{2\log(1/\delta)}n}
\]
by \cite[Theorem 26.12]{livni2014computational}.
On the other hand, we note that
\begin{align*}
\|\nabla u_{(a,W,b)}(x)\|^2 &= \left\|\sum_{i=1}^m a_i 1_{\{x:w_i\cdot x+b_i>0\}}w\right\|^2
	= \sum_{i,j=1}^m a_ia_j\,1_{\{x: w_i\cdot x+b_i>0\}}1_{\{x:w_j\cdot x+b_j>0\}}w_i\cdot w_j
\end{align*}
and thus, denoting by $\hat\pi_N$ the empirical measure of $x_1,\dots, x_N$, we find that
\[
\left|\int_R\|\nabla u\|^2 \dx - |R| \int_R\|\nabla u\|^2 \d\hat \pi_N\right| \leq |R|\sum_{i,j=1}^m |a_i|\,\|w_i\|\,|a_j|\,\|w_j\| \left|\int_R 1_{\{w_i\cdot x+b_i>0\text{ and }w_j\cdot x+b_j>0\}}\d(\pi_R - d\hat\pi_N)\right|.
\]
It is well-known that half-spaces form a learnable hypothesis class of finite VC-dimension, and as do intersections of two half-spaces by a common boosting bound \cite[Theorem 10.3]{livni2014computational}. Thus, by the fundamental theorem of PAC learning \cite[Theorem 6.8]{livni2014computational}, there exists a universal constant $C>0$ such that
\[
\left|\int_R 1_{\{w_i\cdot x+b_i>0\text{ and }w_j\cdot x+b_j>0\}}\d(\pi_R - d\hat\pi_N)\right| \leq C\sqrt{\frac{1+ \log(1/\delta)}N}
\]
where the constant $C$ also absorbs the VC-bound. Overall, we conclude that
\[
\left|\int_R\|\nabla u_{(a,W,b)}\|^2 \dx - |R| \int_R\|\nabla u_{(a,W,b)}\|^2 \d\hat \pi_N\right| \leq C\, |R|\left(\sum_{i=1}^m|a_i|\,\|w_i\|\right)^2 \sqrt{\frac{\log(1/\delta)}N}.
\]
Passing to the limit $m\to\infty$, the same bound holds for general Barron functions with the Barron norm in place of the norm of the weights and biases $(a,W,b)$.
 
{\bf Step 3. Constructing an energy competitor.} Let $\eps = 1/m$. By Theorem \ref{theorem main}, we know that there exists $u_\eps$ such that
\[
\|u_\eps\|_\B\leq C \,\|g\|_\B\,|\log\eps|, \qquad u_\eps \equiv g\text{ on }\partial R, \qquad \|\nabla (u^*-u_\eps)\|_{L^2(\Omega)}^2 \leq C\,\|g\|_\B \eps.
\]
By the direct approximation theorem for Barron functions (see Section \ref{section review}), there exist $(a,W,b)$ such that
\begin{enumerate}
\item $\|a\|^2 + \|W\|^2 + \|b\|^2 \leq \|u_\eps\|_\B \leq C \,\|g\|_\B\,|\log\eps| = C\|g\|_\B \log(m)$ and
\item $\|u_{(a,W,b)} - u_\eps\|_{H^1(\Omega)} \leq C\,\|u_\eps\|_\B \,m^{-1/2} \leq C\,\|g\|_\B(\log m)\,m^{-1/2}$.
\end{enumerate}
Using the trace operator from $H^1(R)$ to $L^2(\partial R)$, we have
\begin{align*}
\int_{\partial R} \big|g- u_{(a,W,b)}\big|^2 \ds &= \|u_\eps - u_{(a,W,b)}\|_{L^2(\partial R)}^2
	\leq C\,\|u_\eps - u_{(a,W,b)}\|_{H^1(R)}^2 \leq C\,\|g\|_\B^2|\log m|^2\,m^{-1}
\end{align*}
for the boundary penalty term and
\begin{align*}
\frac12 \int_R \|\nabla u_{(a,W,b)}\|^2\dx &= \frac12 \int_R \big|\nabla \big(u_{(a,W,b)} - u^*+u^*\big)\big\|^2\dx\\
	&\leq  \frac12\int_R \frac{1+\alpha}2\|\nabla u^*\|^2 + \frac{1+\alpha^{-1}}2 \|\nabla (u_{(a,W,b)} - u^*)\|^2\dx\\
	&\leq \frac{1+\alpha}2 \int_R \|\nabla u^*\|^2\dx + \frac12\left(1+\frac1\alpha\right) \|u_{(a,W,b)} - u^*\|_{H^1(\Omega)}^2\\
	&\leq \frac{1+\alpha}2 \int_R \|\nabla u^*\|^2\dx + \left(1+\frac1\alpha\right)  \big(\|u_{(a,W,b)} - u_\eps\|_{H^1(\Omega)}^2 + \|u_\eps - u^*\|_{H^1(\Omega)}^2\big)\\
	&\leq \frac{1+\alpha}2 \int_R\|\nabla u^*\|^2\dx + \frac {C\,\|g\|_\B^2}\alpha \left(\frac{|\log m|^2}m + \frac1{m^2}\right)
\end{align*}
for $\alpha>0$ which can be chosen later. By a union bound, we find that all bounds hold simultaneously with probability at least $1-2\delta$ and
\begin{align*}
\inf_{(a,W,b)} \widehat E(a,W,b) &\leq (1+\alpha) \,E_{Dir}(u^*) + C\,\|g\|_\B^2 \left(\left(\alpha^{-1}+\lambda\right)\,\frac{|\log m|^2}{m} + \sqrt{\frac{\log(1/\delta)}N } + \lambda\sqrt{\frac{\log(1/\delta)}n}\right) \\
	&\qquad+ C\mu\,\|g\|_\B \log m\label{eq deep ritz energy bound}\showlabel
\end{align*}
with a domain-dependent constant $C>0$.

{\bf Step 4. Norm bound.} Assume that $(a,W,b)$ are parameters which satisfy the energy bound \eqref{eq deep ritz energy bound}. Then
\begin{align*}
\|a\|_2^2 + \|W\|_F^2 + \|b\|^2_2
	&\leq \frac{\widehat E(a,W,b) - E_{Dir,n}(u_{(a,W,b)})  - \lambda\,E_{bd,n}(u_{(a,W,b)})}\mu\\
	&\leq \frac{\widehat E(a,W,b) - E_{Dir}(u_{(a,W,b)})  - \lambda\,E_{bd}(u_{(a,W,b)})}\mu\\
	&\qquad\qquad + \frac{C}\mu \big(\|a\|_2^2 + \|W\|_F^2 + \|b\|_2^2\big)^2\left(\sqrt{\frac{\log(1/\delta)}{N}} +\sqrt{\frac{\log(1/\delta)}n}\right).
\end{align*}
We distinguish between two cases: 
\[
E_{Dir}(u_{(a,W,b)})  + \lambda\,E_{bd}(u_{(a,W,b)}) \leq E_{Dir}(u^*).
\]
This case will be studied below in Step 5b. Otherwise, we have
\begin{align*}
\|a\|_2^2 + \|W\|_F^2 + \|b\|^2_2
	 &\leq \frac{\widehat E(a,W,b) - E_{Dir}(u^*) }\mu  + \frac{C}\mu \big(\|a\|_2^2 + \|W\|_F^2 + \|b\|^2\big)^2\left(\sqrt{\frac{\log(1/\delta)}n} + \sqrt{\frac{\log(1/\delta)}N}\right)\\
	&\leq \frac{C\,\|g\|_\B^2}\mu \left( \alpha + \big(\alpha^{-1}+\lambda\big) \frac{\log^2m}m + \sqrt{\frac{\log(1/\delta)}N} + \sqrt{\frac{\log(1/\delta)}n}\right) + C\,\|g\|_\B\,\log m\\
	&\qquad\qquad + \frac{C}\mu \left(\sqrt{\frac{\log(1/\delta)}n} + \sqrt{\frac{\log(1/\delta)}N}\right)\big(\|a\|_2^2 + \|W\|_F^2 + \|b\|^2\big)^2
\end{align*}
since $E_{Dir}(u^*) \lesssim \|g\|_\B^2$. To minimize the bound, we select $\alpha = \log m / \sqrt m$, and we note that the bound does not become worse under a choice $\lambda = \alpha^{-1} = \sqrt m/\log m$. By assumption, we have
\[
\frac{\|g\|_\B}{\mu_m} \left(\frac{\log m}{\sqrt m} + \sqrt{\frac{\log(1/\delta)}{N_m}} +  \sqrt{\frac{\log(1/\delta)}{n_m}}\right) \leq \log m
\]
for large enough $m$, such that the bound 
\[
\|u_{(a,W,b)}\|_\B \leq \|a\|^2+ \|W\|^2 + \|b\|^2\leq C\,\|g\|_\B\log m + \frac{\sqrt{\log(1/\delta)}}\mu(n^{-1/2}+N^{-1/2})\big(\|a\|^2+ \|W\|^2 + \|b\|^2\big)
\]
 holds for sufficiently large $m$ at the expense of a larger constant $C$. The bound has the structural form $z \leq B+ \eps z^2$, which allows for a dichotomy of solutions
 \[
 z \leq \frac1\eps\left(\frac12 - \sqrt{\frac14 - C\eps}\right) \approx C + 2C^2\eps\qquad\text{or}\quad z \geq \frac1\eps\left(\frac12 + \sqrt{\frac14 - C\eps}\right)\approx \frac1\eps - C.
 \]
 Since the a priori bound
 \[
  \|a\|^2+ \|W\|^2 + \|b\|^2 \leq \frac{\widehat E(a,W,b)}\mu \leq \frac C\mu \leq C\sqrt{m}
 \]
holds for a suitable constant $C$, we can exclude the upper solution interval which would require
 \[
   \|a\|^2+ \|W\|^2 + \|b\|^2 \geq c \frac{1}{\sqrt{m/n} + \sqrt{m/N}} \gg \sqrt{m} \qquad \text{since}\quad n, N\gg m^2.
 \]
 Thus, at the expense of a yet larger constant, we have
 \[
 \|u_{(a,W,b)}\|_\B  \lesssim \|g\|_\B\log m.
 \]

{\bf Step 5a. Closeness based on generalization bounds.} Assume that \eqref{eq deep ritz energy bound} holds for $(a,W,b)$ and that $\|u_{(a,W,b)}\|_\B \leq C\|g\|_\B\log m$. With probability at least $1-2\delta$, we have
\begin{align*}
E_\lambda(u_{(a,W,b)})&\leq \widehat E(a,W,b) + C\|g\|_\B^2\log^2m\left(\sqrt{\frac{\log(1/\delta)}n} + \sqrt{\frac{\log(1/\delta)}N}\right)\\
	&\leq E_{Dir}(u^*) + C\,\|g\|_\B^2 \left(\frac{\log m}{\sqrt m} + \log^2m\sqrt{\frac{\log(1/\delta)}N} + \log^2m \sqrt{\frac 1n}\right)\\
	&\qquad\qquad + C\mu \|g\|_\B \log m\\
	&\leq E_{Dir}(u^*) + C\sqrt{(\log1/\delta)}\,\big(\|g\|_\B^2+1\big) \frac{\log m}{\sqrt m}.
\end{align*}
We conclude by Step 1, using the bound
\begin{align*}
\|u_{(a,W,b)} - u^*\|_{H^1(\Omega)}^2 &\lesssim E_\lambda(u_{(a,W,b)}- E_{Dir}(u^*) + \sqrt{\frac{E_\lambda(u_{(a,W,b)}- E_{Dir}(u^*) }\lambda}\\
	&C\sqrt{(\log1/\delta)}\,\big(\|g\|_\B^2+1\big) \frac{\log m}{\sqrt m}. 
\end{align*}

{\bf Step 5b. Closeness without Barron norm control.}
If the bound 
\[
E_{Dir}(u_{(a,W,b)})  + \lambda\,E_{bd}(u_{(a,W,b)}) \leq E_{Dir}(u^*)
\]
holds, we use Step 1 to deduce that
\begin{align*}
\|\nabla(u_{(a,W,b)} - u^*)\|_{L^2(R)}^2 &\leq E_{Dir}(u_{(a,W,b)}) - E_{Dir}(u^*) + 2\,\|u-u^*\|_{L^2(\partial R)} \|\nabla u^*\|_{L^2(\partial R)}\\
	&\leq  2\,\|u-u^*\|_{L^2(\partial R)} \|\nabla u^*\|_{L^2(\partial R)} - \lambda \,\|u-u^*\|_{L^2(\partial R)}^2\\
	&= \big(2 \|\nabla u^*\|_{L^2(\partial R)} - \lambda\|u-u^*\|_{L^2(\partial R)}\big)\,\|u-u^*\|_{L^2(\partial R)}.
\end{align*}
As the left is non-negative, so is the right, meaning that 
\[
\|u-u^*\|_{L^2(\partial R)} \leq \frac{\|\nabla u^*\|_{L^2(\partial R)}}\lambda.
\]
Inserting this once more in the bound of Step 1, we deduce that
\begin{align*}
\|u_{(a,W,b)} - u^*\|_{H^1(R)}^2 &\leq E_{Dir}(u_{(a,W,b)}) - E_{Dir}(u^*) + 2\,\|u-u^*\|_{L^2(\partial R)} \|\nabla u^*\|_{L^2(\partial R)}\\
	&\leq \frac{2\,\|\nabla u^*\|_{L^2(\partial R)}}\lambda \leq C\,\frac{\|g\|_\B^2}\lambda = C\,\frac{\|g\|_\B^2\,\log m}{\sqrt m}.\qedhere
\end{align*}
\end{proof}

A few observations are in order.

\begin{remark}[Integration points]
We note that the exact same result holds for any choice of integration points $x_i, z_j$ such that the uniform rates of convergence
\[
\max_{\|\nu\|^2+b^2 =1} \left|\int_{\partial R} \sigma(\nu\cdot x+b) - \frac{|\partial R|}n\sum_{j=1}^n \sigma(\nu\cdot z_j+b)\right| \lesssim n^{-1/2}
\]
and
\[
\max_{(\nu_1, b_1), (\nu_2, b_2) \in \R^3\times \R^3} \left| \int_R \chi_{\{\nu_1\cdot x+b_1>0\text{ and }\nu_2\cdot x+b_2>0\}}\dx - \frac{|R|}{N} \sum_{i=1}^N1_{\{\nu_1\cdot x_i+b_1>0\text{ and }\nu_2\cdot x_i+b_2>0\}}\right|\lesssim N^{-1/2}
\]
hold.
\end{remark}

\begin{remark}[Physics-informed neural networks]
While the Deep Ritz method approaches a PDE through its variational characterization, the method of physics-informed neural networks (PINNs) takes a more direct approach of minimizing a computational approximation to a penalty of the strong formulation
\[
E_{PINN}(u) = \int_\Omega \big|\Delta u + f\big|^2\dx + \frac\lambda 2\int_{\partial\Omega} \big|u-g\big|^2\d A.
\]
In Barron spaces, we cannot obtain an analogous result to Theorem \ref{theorem deep ritz} for PINNs since the ReLU activation is non-smooth and solutions $u^*$ with Barron data fail to lie in the natural energy space $H^2(\Omega)$. The direct approximation theorem for Barron functions holds in $H^1(\Omega)$, but not $H^2(\Omega)$. However, we conjecture that analogous methods could be employed using higher order activation functions, building on the results of \cite{vaishampayan2024solving}.
\end{remark}

\section{Proof of Theorem \ref{theorem boundary}: Boundary data in Barron space} \label{section boundary}

\begin{proof}[Proof of Theorem \ref{theorem boundary}]
If $g\in\B(\R^2)$, then trivially $g(x,\bar y)\in \B(\R)$ for any $\bar y\in \R$ by definition since the linear weight of the second variable is absorbed into the bias term. More quantitatively, we see that the estimate $\|g(\cdot, \bar y)\|_{\B(\R)}\leq (1+ |\bar y|)\|g\|_{\B(\R^2)}$ holds. It therefore suffices to prove the opposite inclusion.

By smoothness, functions of the form $(x,y)\mapsto xy$ are locally Barron (but not globally Barron since they grow quadratically at infinity). For us, it suffices that there exists $\tilde h\in \B$ such that $\tilde h(x,y) = xy$ for all $(x,y)\in R$. Now consider
\[
\tilde g(x,y) = g(x,y) + c_{00}\,\frac{a-x}a\cdot\frac{b-y}b + c_{10}\,\frac xa\cdot\frac{b-y}b + c_{01}\,\frac{a-x}a\cdot\frac yb+ c_{11}\,\frac xa \cdot\frac yb.
\]
Since smooth functions are (locally) Barron, we see that $\tilde g\in \B(R)$ if and only if $g\in\B(R)$. Choosing $c_{ij} = -g(ai,bj)$, we may assume without loss of generality that $g(0,0) = g(0,b) = g(a,0) = g(a,b) = 0$. 

Denote $g^y_1(x) = g(x,b)\in \B(\R)$. Without loss of generality we may assume that $g^y_1(x) = 0$ if $x\leq 0$ or $x\geq a$. Effectively a (continuous) continuation by a constant corresponds to setting the derivative to zero beyond a certain point, or equivalently to adding a Dirac mass to the second derivative measure with a weight which is bounded by the Lipschitz constant of the function (and thus, in particular by the Barron norm).

Recall that the function $\tilde h(x,y) =  \frac{y}b \,g^y_1(b\cdot x/y)$ is a Barron function of two variables by the argument presented in the first step of the proof of Lemma \ref{lemma barron regularity}. Furthermore, we see that
\[
\tilde h(x,b) = 1 \cdot g^y_1(1\cdot x) = g^y_1(x) = g(x,b), \qquad \tilde h(0,y) = y/b \cdot g^y_1(0) = 0
\]
and 
\[
\tilde h(a,y) = \frac yb \cdot g^y_1\left(a\,\cdot\frac{b}y\right) = 0
\]
since $b/y\geq 1$ and $g^y_1(x) =0$ if $x\geq a$. Furthermore, note that 
\[
\tilde h(x,y) = \frac{y}b\, g^y_1\left(\frac by\,x\right) = 0
\]
if either $x=0$ and $y>0$ is arbitrary or $y> \frac ba\,x$.
 We can fit data on all four sides of the rectangle by adding four separate Barron functions of two variables in this manner. 
\end{proof}

\section{Discussion: Regularity theory, Deep Ritz, PINNs and finite elements}\label{section conclusion}

We have demonstrated that in general, the unique harmonic function in a domain with boundary values that can be represented by a ReLU network with a single hidden layer is not Lipschitz continuous. This is true even if the boundary condition is given by a neural network with a single neuron in a single hidden layer. The harmonic function is therefore not an element of any function class in which the norm controls the Lipschitz constant, such as Barron spaces or tree-like spaces for multi-layer networks \cite{deep_barron}. Making the neural network deeper and even infinitely wide does not enable it to represent exactly, but only to approximate the harmonic function. Along an approximating sequence of ReLU networks, it does not suffice to make the architecture more complex or increase the number of parameters without also allowing the average magnitude of coefficients to grow to infinity in a fashion which allows the Lipschitz constant to go to infinity. The solution can be approximated highly efficiently with shallow networks of slowly growing Barron norm.

{\bf Implications for Optimization.} Both for the Deep Ritz method and the PINN method, the task is ultimately to minimize an energy functional on Barron space which can be made $\eps$-small with Barron norm $\sim|\log\eps|$. Thus while the loss functionals do not have minimizers, they have exponentially decaying tails. This resembles the situation of classification problems with logistic or cross-entropy loss, which can be efficiently solved by gradient descent methods. In fact, solutions to the gradient flow $\dot z = - e'(z)$ for the simple toy model $e(z) = \exp(-z)$ exactly achieve the $e(z(t)) = O(1/t)$ rate which is guaranteed for gradient flows of convex functions {\em with} minimizers.

Our results can be read to suggest that there are no obstacles to training directly from the approximation theory in the fashion of \cite{dynamic_cod} based on \cite{relutraining, approximationarticle}, but they cannot be taken as positive guarantees that training is possible.

{\bf Benchmark problems and deeper networks.} Given the extremely high degree of approximation at a relatively low Barron norm, it appears that deeper neural networks are not required in this context, at least not from the approximation perspective. In this way, the problem of solving the Poisson equation on a rectangular domain with regular right hand side and Barron boundary data may be `too easy' as a benchmark problem for neural PDE solvers, or at least one in which the greater flexibility of deep networks may not be visible.

On the other hand, it can be mostly ruled out that deep networks have an advantage over two-layer networks from their ability to represent the solution of the PDE. In this sense, the problem provides a non-trivial opportunity to study whether deeper networks enjoy an advantage from the training perspective: If they do better, that would have to be the reason.

In \cite{wojtowytsch2020some}, evidence is presented that deeper neural networks may be beneficial for {\em non-linear} PDEs. 

{\bf Different activation functions.} The spatial gradients of shallow ReLU networks are piecewise constant functions, which makes gradient based training impossible. This lack of regularity can be remedied without changing the underlying function spaces by selecting a smoother activation function such as $\sigma_\delta(z) = \delta\,\log(1+\exp(z/\delta))$ which satisfies
\[
\lim_{\delta\to 0^+} \sigma_\delta = \sigma\quad\text{locally uniformly}, \qquad 
\lim_{\delta\to 0^+} \sigma_\delta'(z) = \sigma'(z)\quad\forall\ z\neq 0, \qquad 
\|\sigma_\delta'\|_\infty \leq 1
\] 
and $\sigma_\delta'' \wto\sigma''$ weakly in the sense of Radon measures. Higher powers of the ReLU activation are smoother and can be studied using similar tools to this article -- see \cite{vaishampayan2024solving}.

\section*{Acknowledgements}

The author gratefully acknowledges support from the National Science Foundation through grant NSF DMS 2424801.

\appendix

\section{Additional proofs}

We conclude with results about Barron space, (fractional order) Sobolev spaces and their relationship, as well as elliptic regularity theory for the Laplacian on a rectangular domain. These results were used in the main text and we believe them to be well-known to the experts, but we present them in detail for the reader's convenience. 

We begin with the compact embedding theorem outlined in Section \ref{section barron review}.

\begin{lemma}\label{compact embedding theorem}
Let $u_n\in \B$ be a sequence of Barron functions such that $\|u_n\|_\B\leq C$. Then there exists $u\in \B$ such that $\|u\|_\B\leq C$ and a subsequence $u_{n_k}$ of $u_n$ such that $u_{n_k}\to u$ as $k\to\infty$:
\begin{enumerate}
\item Uniformly on all compact subsets of $\R^d$,
\item Weakly in $W^{1,p}(\Omega)$ for all bounded open sets $\Omega$ and $p<\infty$, and
\item Strongly in $L^p(\mu)$ for any measure $\mu$ with finite $p$-th moments.
\end{enumerate}
\end{lemma}

\begin{proof}[Sketch of Proof]
We note that $|u_n(0)| \leq \|u_n\|_\B$ and $[u_n]_{Lip}\leq \|u_n\|_\B$. 

By a corollary to the Arzela-Ascoli theorem, for any $R>0$, the space $C^{0,1}(\overline{B_R(0)}) = W^{1,\infty}(B_R(0))$ embeds compactly into $C^0(\overline {B_R(0)})$ equipped with the supremum norm, i.e.\ for every $R>0$, there exists $u^R \in C^0(\overline{B_R(0)})$ and a subsequence $u_{n_k}$ of $u_n$ such that $u_{n_k}\to u$ uniformly on $\overline{B_R(0)}$. Using a standard diagonal sequence argument and a sequence of growing radii like $R_n = n$, we can show that there exists $u^\infty$ and a subsequence $u_{n_k}$ of $u_n$ such that $u_{n_k}\to u$ uniformly on $\overline{B_{R_n}(0)}$ for every $n\in \mathbb N$ and thus on every compact subset of $\R^n$. 

The uniform limit of Lipschitz functions can easily be seen to be Lipschitz-continuous, i.e.\ $u$ is also Lipschitz-continuous. 

Similarly, $W^{1,\infty}(B_R(0))$ embeds continuously into $W^{1,p}(B_R(0))$ for any $p<\infty$. Since bounded subsets of $W^{1,p}(B_R(0))$ are weakly sequentially compact, we can extract weakly convergent subsequences as well. Additionally, if $v_n \wto v$ in $W^{1,p}(B_R(0))$, then also $v_n\wto v$ in $W^{1,q}(B_R(0))$ for $q<p$. In this way, choosing a sequence of exponents like $p_n= n$, we can use the same diagonal sequence argument also here. 

Assume that $\mu$ is a measure with finite $p$-th moments. Then $u_{n_k}\to u$ pointwise almost anywhere by the first assertion, and
\[
|u_{n_k}(x)| \leq |u_{n_k}(0)| + [u_{n_k}]_{Lip}|x| \leq C(1+|x|).
\]
The claim follows by the dominated convergence theorem. 

It remains to show that $u\in \B$. This follows as e.g.\ in \cite{barron_new} from the compactness theorem for Radon measures \cite[Chapter 1.9]{evans2015measure}.
\end{proof}

Recall that the norm of the Sobolev space $H^s(\R^d)$ of order $s\geq 0$ can be defined as
\[
\int_{\R^d}|\hat f|^2(\xi)\,\big(1+|\xi|^2\big)^{s/2}\d\xi
\]
where $\hat f$ is the Fourier transform of $f:\R^d\to\R$. Other equivalent definitions exist -- see e.g.\ \cite{di2012hitchhikers, leoni2023first} for an introduction to fractional order spaces. On bounded domains, we use the definition
\[
\|u\|_{H^{1/2}(\partial\Omega)} = \|u\|_{L^2(\partial\Omega)} + \inf_{v\in H^1(\Omega) \text{ s.t. } v= u\text{ on }\partial\Omega} \|\nabla v\|_{L^2(\Omega)}
\]
for the fractional Sobolev space on the boundary of a bounded domain $\Omega$. For the boundary equality, consider $v$ on $\partial\Omega$ as the trace of the Sobolev function -- see also \cite[Proposition 14.40 and Theorem 15.20]{leoni2017first} or \cite[Satz 9.40]{dobrowolski2010angewandte} for equivalence.

The key ideas of embeddings between Barron space and fractional Sobolev spaces build on \cite[Section IX]{barron1993universal}. We will choose the normalization for the Fourier transform such that
\begin{equation}\label{eq fourier representation}
f(x) = \int_{\R^d}\hat f(\xi)\,e^{2\pi i\langle \xi,x\rangle}\d\xi.
 \end{equation}

\begin{lemma}\cite[Proposition 4.1 and Lemma 7.1]{barron_boundaries}\label{lemma sobolev embedding}
Let $\Omega\subseteq\R^d$ be a bounded open set. Then for every $f\in H^{d/2+1+s}(\R^d)$ with $s>0$, there exists $F\in \B$ such that $F\equiv f$ in $\Omega$ and
\[
\|F\|_{\B} \leq 2\pi^2\,\omega_d^{1/2}\left(\frac{s+1}{s}\right)^{1/2} \left(8+ {\diam(\Omega)^2}\right)  \|f\|_{H^{d/2+1+s}(\R^d)},
\]
where $\omega_d$ denotes the volume of the unit ball in dimension $d$.
\end{lemma}

\begin{proof}
By \cite[Lemma 7.1]{barron_boundaries}, we have
\[
\|f\|_\B \leq C\int_{\R^d} |\hat f|(\xi)\big(1+|\xi|^2\big)\d\xi.
\] 
This can be proved as follows:

\begin{enumerate}
\item The function $e_R(z) = \exp(-2\pi i Rz)$ is a (complex-valued) Barron function on an interval $[-A, A]$ since
\[
\int_{-A}^A |e_R''|(z)\,\sqrt{1+z^2}\dz = (2\pi R)^2 \int_{-A}^A \sqrt{1+z^2}\dz \leq 2 (2\pi R)^2 \left(1 + \frac{A^2}2\right).
\]
It can be extended as a linear function outside $[-A,A]$ without increasing the norm.

\item By the same argument, for any $\xi\in \R^d$, the function $e_\xi:\R^d\to\R$, $e_\xi(x) = \exp (-2\pi i \langle \xi, x\rangle)$ is a (complex-valued) Barron function on the {\em bounded} set $\Omega$ and norm at most 
\[
\|e_\xi\|_\B \leq 8\pi^2 \,\|\xi\|^2\left(1 + \frac{\diam(\Omega)^2}8\right).
\] 

\item On $\Omega$, the function $f$ can be represented as a superposition of the functions $e_\xi$ as in \eqref{eq fourier representation}.
\end{enumerate}

Technical details can be found in the reference. The integral on the right can be estimated by
\begin{align*}
\int |\hat f|(\xi)\big(1+|\xi|^2\big)\d\xi &= \int |\hat f|(\xi)\big(1+|\xi|^2\big)^{1+\alpha/2-\alpha/2}\d\xi\\
	&\leq \left(\int |\hat f|^2(\xi)\big(1+|\xi|^2\big)^{2+\alpha}\d\xi\right)^\frac12 \left(\int (1+|\xi|^2)^{-\alpha}\d\xi\right)^\frac12\\
	&= [f]_{H^{1+\alpha/2}} \left(\omega_d\int_0^\infty \frac{r^{d-1}}{(1+r^2)^\alpha}\dr\right)^\frac12\\
	&\leq \omega_d^{1/2}  [f]_{H^{1+\alpha/2}}\left(1+ \int_1^\infty r^{d-1-\alpha}\dr\right)^\frac12\\
	&\leq \sqrt{\omega_d+\frac{\omega_d}{\alpha-d}} \,[f]_{H^{1+\alpha/2}},
\end{align*}
which is finite if $\alpha>d$.
\end{proof}

In order to work primarily with integer order Sobolev spaces, we recall the following interpolation theorem.

\begin{lemma}\label{lemma sobolev interpolation}
Let $s\in (s_1, s_2)$. Then
\[
\|f\|_{H^s(\R^d)} \leq \|f\|_{H^{s_1}}^{\frac{s_2-s}{s_2-s_1}}\|f\|_{H^{s_2}}^{\frac{s-s_1}{s_2-s_1}}\qquad \forall\ f\in H^s(\R^d).
\]
\end{lemma}

\begin{proof}
Let $\gamma = \frac{s_2-s}{s_2-s_1}\in(0,1)$ and note that $1-\gamma = \frac{s-s_1}{s_2-s_1}\in(0,1)$. Then
\begin{align*}
\|f\|_{H^s(\R^d)}^2 &= \int_{\R^d} (1+|\xi|^2)^{s}\,|\hat f|^2(\xi)\,\d\xi\\
	&= \int_{\R^d}(1+|\xi|^2)^{\alpha s}\,|\hat f|^{2\gamma}(\xi) \cdot (1+|\xi|^2)^{(1-\alpha)s}\,|\hat f|^{2(1-\gamma)}(\xi)\,\d\xi\\
	&\leq \left( \int_{\R^d}(1+|\xi|^2)^{\alpha s/\gamma}\,|\hat f|^{2}(\xi)\,\d\xi\right)^\gamma \left( \int_{\R^d}(1+|\xi|^2)^{(1-\alpha)s/(1-\gamma)}\,|\hat f|^{2}(\xi)\,\d\xi\right)^{1-\gamma}\\
	& = \|f\|_{H^{\alpha s/\gamma}}^{2\gamma} \|f\|_{H^{(1-\alpha) s/(1-\gamma)}}^{2(1-\gamma)}.
\end{align*}
Select $\alpha = \frac{s_1}s\gamma\in(0,\gamma)\subset (0,1)$ and note that $\frac{\alpha s}{\gamma} = \frac{\frac{s_1}s\gamma\,s}\gamma = s_1$ and
\[
 \frac{(1-\alpha)s}{1-\gamma} 
	= \frac{\frac{s-s_1\gamma}s\,s}{\frac{s-s_1}{s_2-s_1}} 
	= \frac{(s_2-s_1)\left(s-s_1\,\frac{s_2-s}{s_2-s_1}\right)}{s-s_1}
	= \frac{s(s_2-s_1) - s_1(s_2-s)}{s-s_1} = \frac{s_2(s-s_1)}{s-s_1} = s_2.\qedhere
\]
\end{proof}

In particular, if $s\in(2,3)$, then
\[
\|f\|_{H^s(\R^d)} \leq \|f\|_{H^2(\R^d)}^{3-s}\|f\|_{H^3(\R^d)}^{s-2}\qquad \forall\ f\in H^s(\R^d).
\]
As we only have access to a function on a domain $\Omega$ at times, we employ the following degree-independent Sobolev extension theorem for domains which are merely Lipschitz-regular.

%Elias Stein
\begin{theorem}\cite[Chapitre VI, 4.3, Cinqui\`eme Th\'eor\`eme]{stein1967integrales}
Let $\Omega\subseteq\R^d$ be a bounded Lipschitz-domain. There exists a linear extension operator $T$ such that for all $k\in \N$ and $1\leq p\leq \infty$ the operator $T:W^{k,p}(\Omega) \to W^{k,p}(\R^d)$ is continuous with a norm bound which depends on $n, k, p$ and the Lipschitz-constant of the domain $\Omega$.
\end{theorem}

Providing even a sketch of the proof is unfortunately beyond the scope of this exposition. By Stein's extension theorem and Sobolev space interpolation, we can extend a function $u\in H^2(\Omega)\cap H^3(\Omega)$ to $Tu:\R^d\to\R$ such that
\[
\|Tu\|_{H^s(\R^d)} \leq \|Tu\|_{H^3(\R^d)}^{s-2}\|Tu\|_{H^2(\R^d)}^{3-s} \leq C\,\|u\|_{H^3(\Omega)}^{s-2}\|u\|_{H^2(\Omega)}^{3-s}. 
\]
Next, we present a higher order regularity theorem for the Laplace operator on a rectangular domain.

\begin{lemma}\label{lemma regularity cube}
Let $\Omega=\prod_{i=1}^d(0, a_i)$ be a $d$-dimensional cuboid and $f\in H^k(\Omega)$, $g\in H^{k+2}(\Omega)$ where $k\in \N_0$ and $H^0 = L^2$. Consider the solution $u$ of the PDE
\[
\begin{pde} 
\Delta u &= f &\text{in }\Omega\\
u &= g &\text{on }\partial \Omega.
\end{pde}
\]
Then $\|u\|_{H^{k+2}(\Omega)} \lesssim \|f\|_{H^k(\Omega)} + \|g\|_{H^{k+2}(\Omega)}$.
\end{lemma}

Theorems of this type are classical and well-known for domains with $C^k$-smooth boundary \cite[Theorem 8.12]{gilbarg2015elliptic}. By contrast, even $H^2$-regularity is known to fail in less regular domains. For instance, if we define in polar coordinates 
\[
\Omega_\alpha = \{x\in\R^2 : 0<r<1, \:0 < \theta < \alpha\}, \qquad u_\alpha(r,\theta) = r^{\pi/\alpha} \sin\left(\frac{\pi}\alpha\,\theta\right)
\]
then
\[
\Delta u_\alpha = \left(\partial_r\partial_r + \frac1r\,\partial_r + \frac1{r^2}\partial_\theta\partial_\theta\right) u_\alpha = \left\{\frac\pi\alpha\left(\frac\pi\alpha-1\right) +\frac\pi\alpha - \left(\frac{\pi}{\alpha}\right)^2\right\}r^{\pi/\alpha\,-2}\sin\left(\frac{\pi}\alpha \,\theta\right) =0.
\]
The domain $\Omega_\alpha$ is convex if and only if $\alpha \leq \pi$. Since $u_\alpha$ is positively homogeneous of degree $\pi/\alpha$, the tensor of $k$-th derivatives of $u_\alpha$ is positively homogeneous of degree $\pi/\alpha - k$. This is $p$-integrable on $\Omega_\alpha \subset \R^2$ if and only if $p(\pi/\alpha - k)> -2$ (unless $u_\alpha$ is a homogeneous polynomial and the derivatives are zero). If $\frac\pi\alpha$ is not an integer, then $u_\alpha \in H^k(\Omega_\alpha)$ if and only if $\frac\pi\alpha >  k - \frac{2}p$. In particular, for $\pi/2 < \alpha <\pi$, we find that $u_\alpha\notin H^3(\Omega_\alpha)$. As $\alpha\nearrow 2\pi$, we find that $u_\alpha \in W^{2,4/3}(\Omega_\alpha)$ (and by a similar consideration that $u_\alpha \in H^{3/2}$). $H^2$-regularity clearly fails.

By construction $u_\alpha \equiv 0$ on the straight segments of the boundary. This means that $u_\alpha \equiv \tilde u_{\alpha, m} = r^m u_\alpha$ on $\partial\Omega_\alpha$. Choosing $m$ large, we can make $\tilde u_{\alpha,m}$ as smooth as we desire. This suggests to us that the method of proof of Lemma \ref{lemma rectangle} as it uses Lemma \ref{lemma regularity cube} may be fragile even when passing to a non-rectangular parallelogram. Similarly, we anticipate additional technical difficulty when considering on a rectangular domain a general constant coefficient second order elliptic differential operator $L = - a^{ij}\partial_i\partial_j$ whose symbol $a^{ij}\xi_i\xi_j$ has eigenvectors which are not parallel to the coordinate axes.

Cuboid domains are special for the Laplacian since its eigenfunctions can be chosen to be separable (i.e.\ they decompose into a product of functions of the $d$ orthogonal coordinates $x_1,\dots, x_d$. We believe Lemma \ref{lemma regularity cube} to be well-known to the experts, but have been unable to locate a reference for it.

\begin{proof}[Proof of Lemma \ref{lemma regularity cube}]
We note that $\tilde u = u-g$ satisfies $\tilde u \in H_0^1(\Omega)$ and $\Delta \tilde u = f-\Delta g\in H^k(\Omega)$, i.e.\ it suffices to consider the inhomogeneous Poisson equation with homogeneous Dirichlet boundary condition.\footnote{\ We prescribe the boundary condition by requiring that $u- g\in H_0^1$ or equivalently that the traces of $u$ and $g$ on $\partial\Omega$ coincide.}\ We can expand $\tilde u, f-\Delta g$ in a basis of eigenfunctions of the Dirichlet Laplacian:
\begin{align*}
\tilde u = \sum_{(n_1,\dots,n_d)\in \N^d} \alpha_{n_1,\dots,n_d}\,u_{n_1,\dots,n_d}(x), \qquad f-\Delta g = \sum_{(n_1,\dots,n_d)\in \N^d} \beta_{n_1,\dots,n_d}\,u_{n_1,\dots,n_d}(x)
\end{align*}
The eigenfunctions
\[
u_{n_1,\dots,n_d}(x) = \sin\left(n_1\,\frac{\pi x_1}{a_1}\right) \dots\sin\left(n_d\,\frac{\pi x_d}{a_d}\right), \qquad -\Delta u_{n_1,\dots,n_d} = \pi^2 \sum_{i=1}^d \left(\frac{n_i}{a_i}\right)^2\,u_{n_1,\dots,n_d}
\]
are orthogonal in $H^k$ for all $k\in\N_0$ and infinitely differentiable up to the boundary:
\[
\frac{\partial^{|k|}}{\partial x_1^{k_1}\dots \partial_{x_d}^{k_d}} u_{n_1,\dots,n_d}(x) = \pi^{|k|} \prod_{i=1}^d\left(\frac{n_i}{a_i}\right)^{k_i} \,\sin^{(k_i)}\left(n_d\,\frac{\pi x_i}{a_i}\right)\qquad\forall\ k\in \N^d.
\]
The derivatives $\sin^{(k_i)}\in \{\pm \sin, \pm \cos\}$ are bounded by $1$ in $L^\infty$. From now on, we denote $n=(n_1,\dots,n_d)\in\N^d$ and $\lambda_n = \pi^2\sum_{i=1}^d(n_i/a_i)^2$. In this notation, we observe that
\[
\left\|\frac{\partial^{|k|}}{\partial x_1^{k_1}\dots \partial_{x_d}^{k_d}} u\right\|_{L^\infty(\Omega)} \leq \pi^{|k|}\prod_{i=1}^d\left(\frac{n_i}{a_i}\right)^{k_i}  \leq \lambda_n^{k/2}.
\] 
Based on this bound and orthogonality properties, it is possible to introduce an equivalent norm on $H^k(\Omega)$ by describing $v$ in terms of its expansion coefficients $\alpha_n(v)$:
\[
\|v\|_{H^k}^2 = \sum_{n\in\N^d} \lambda_n^{k}\,\alpha_n(v)^2 \qquad\Ra\quad H^k(\Omega) = \left\{v : \sum_{n\in\N^d}\|n\|^{2k}\,\alpha_n^2(v)<\infty\right\}.
\]
Referring back to the problem at hand, we find that $\alpha_n = \beta_n/\|n\|^2$. In particular
\begin{align*}
\|u-g\|_{H^{k+2}(\Omega)}^2 &= \sum_{n\in\N^d} \lambda_n^{k+2}\,\alpha_n^2 =\sum_{n\in\N^d} \lambda_n^{k+2}\left(\frac{\beta_n}{\lambda_n}\right)^2 = \sum_{n\in\N^d} \lambda_n^k\,\beta_n^2 = \|f-\Delta g\|_{H^k(\Omega)}^2\\
	& \leq 2\big( \|f\|_{H^k(\Omega)}^2 + \|g\|_{H^{k+2}(\Omega)}^2\big)
\end{align*}
and $\|u\|_{H^{k+2}} \leq \|g\|_{H^{k+2}} + \|u-g\|_{H^{k+2}}$. This concludes the proof.
\end{proof}

Finally, we demonstrate that for a two-dimensional rectangle $\Omega = (0,a)\times (0,b)$, any `regular' function $f:\partial\Omega\to\R$ can be extended to a function of the same regularity inside the rectangle.

\begin{lemma}\label{lemma extension from boundary}
Let $\Omega = (0,a)\times (0,b)$ be a rectangle and $I_1, I_2, I_3, I_4$ the `interior' intervals of the four straight segments of $\partial\Omega$. There exists a constant $C>0$ depending on $a,b$, the degree of smoothness $k$ specified below and the precise norm selected on $C^k(I)$ for intervals such that the following is true: If $f\in C^0(\partial \Omega)$ and $f|_{I_j} \in C^k (\overline{I_j})$ for $j=1, \dots, 4$, then there exists $F\in C^k(\Omega)$ such that $F|_{\partial\Omega} = f$ and 
\[
\|F\|_{C^k(\Omega)} \leq C\,\left(\|f\|_{C^0(\partial\Omega)} + \sum_{j=1}^4 \|f\|_{C^k(I_j)}\right).
\]
 The same result holds upon replacing $C^k$ by $C^{k,\alpha}$ for $\alpha\in (0,1]$ or $W^{k,p}$ for $k\in \mathbb N_0, p\in [1,\infty]$.
\end{lemma}

Again, the extension result is well-known for domains with smooth boundaries, but we did not readily find a reference for rectangular domains and choose to provide a proof for the reader's convenience.

\begin{proof}
The proof resembles that of Theorem \ref{theorem boundary}. First, we consider
\[
\tilde f(x_1,x_2) = f(x_1,x_2) - f(a,b)\,\frac {x_1}a \,\frac{x_2}b -  f(a,0) \,\frac{x_1}{a}\,\frac{b-x_2}{b} -  f(0,b) \,\frac{a-x_1}{a}\,\frac{x_2}{b} - f(0,0) \,\frac{a-x_1}{a}\,\frac{b-x_2}{b}
\]
such that $\tilde f(p_j) = 0$ for the corner points $p_1,\dots, p_4$ of the rectangle and
\[
\|\tilde f\|_{C^0(\partial\Omega)} + \sum_{i=1}^4   \|\tilde f\|_{C^i(\overline{I_i})}  \leq (1+C)\, \|\tilde f\|_{C^0(\partial\Omega)} + \sum_{i=1}^4 \|\tilde f\|_{C^i(\overline{I_i})}
\]
for some constant $C>0$ which depends on the precise norm chosen on $C^k(\overline{I_i})$. Since the modifications to $\tilde f$ are affine linear on each side of the rectangle, only the zeroth and first derivatives are affected by the modification. The modification to first derivatives scales as $1/\min\{a,b\}$ in the worst case.

We now define
\[
\tilde F(x_1, x_2) = \frac{x_1}a\,\tilde f(a,x_2) + \frac{a-x_1}a\,\tilde f(0, x_2) + \frac{x_2}b \,\tilde f(x_1, b) + \frac{b-x_2}b\,\tilde f(x_1, 0)
\]
such that $\tilde F$ has the same regularity as $\tilde f|_{\overline{I_j}}$ (or equivalently, as $f|_{\overline{I_j}}$ for $j=1,\dots, 4$ and $\tilde F|_{\partial\Omega} = \tilde f$. The Lemma follows by selecting
\[
F(x_1,x_2) = \tilde F(x_1,x_2) + f(a,b)\,\frac {x_1}a \,\frac{x_2}b + f(a,0) \,\frac{x_1}{a}\,\frac{b-x_2}{b} +  f(0,b) \,\frac{a-x_1}{a}\,\frac{x_2}{b} + f(0,0) \,\frac{a-x_1}{a}\,\frac{b-x_2}{b}.
\]
\end{proof}

\section{Integrals involving Logarithms}\label{appendix integrals}

\begin{lemma}\label{lemma integral estimates}
Let $q>1$. The following estimates hold for some $q_0>0$.
\begin{align*}
\left(\int_0^1 \big|\log r\big|^q\,\dr\right)^{1/q} &\leq  q\\
\left(\int_0^1 \big|\log r\big|^q\,r\dr\right)^{1/q} &\lesssim q\\
\left(\int_0^\infty \frac{\log^q(1+z^2)}{z^2}\dz\right)^{1/q} &\lesssim q\\
\left(\int_0^\infty \frac{\log^q(1+z^2)}{z^3}\dz\right)^{1/q} &\lesssim \frac1{q-1} + q
\end{align*}
\end{lemma}

\begin{proof}
By Stirling's formula $\Gamma(q) \sim \sqrt{2\pi/q}\,(q/e)^q$, we see that $\lim_{q\to \infty} \frac1q\,\Gamma(q+1)^{1/q} = e^{-1}$, so $\Gamma(q+1)^{1/q}\leq q$ for all sufficiently large $q$.

{\bf First integral.} 
\begin{align*}
\int_0^1 |\log r|^q\dr &= \int_1^\infty \left|\log \frac1s\right|^q\,\frac1{s^2}\ds = \int_1^\infty \frac{(\log s)^q}{s^2}\ds = \int_0^\infty \frac{t^q}{e^{2t}}\,e^t\dt = \Gamma(q+1).
\end{align*}

{\bf Second integral.}
\begin{align*}
\int_0^1 \big|\log r\big|^q\,r\dr \leq \int_0^1\big|\log r\big|^q\dr = \Gamma(q+1).
\end{align*}

{\bf Third integral.}
\begin{align*}
\int_0^\infty \frac{\log^q(1+z^2)}{z^2}\dz &= \int_0^1 \frac{\log^q(1+z^2)}{z^2}\dz + \int_1^\infty \frac{\log^q(1+z^2)}{z^2}\dz\\
	&\leq \int_0^1 z^{2(q-1)}\dz + \int_1^\infty\frac{\log^q(2z^2)}{4z^3}\,4z\dz\\
	&\leq 1+ \int_1^\infty\frac{\log^q(\xi)}{\xi^{3/2}}\d\xi\\
	&\leq 1 + \int_0^\infty \frac{\log^q(e^{2t})}{e^{3t}}2e^{2t}\dt  \\
	&= 1 + 2^{q+1}\int_0^\infty t^qe^{-t}\dt\\
	&= 1 + 2^{q+1}\Gamma(q+1)
\end{align*}

{\bf Fourth integral.}
\begin{align*}
\int_0^\infty \frac{\log^q(1+z^2)}{z^3}\dz &= \int_0^1 \frac{\log^q(1+z^2)}{z^3}\dz + \int_1^\infty \frac{\log^q(1+z^2)}{z^3}\dz\\
	&\leq \int_0^1 z^{2q-3}\dz + \int_1^\infty\frac{\log^q(2z^2)}{4z^4}\,4z\dz\\
	&\leq \frac{1}{2(q-1)} + \int_1^\infty\frac{\log^q(\xi)}{\xi^2}\d\xi\\
	&\leq \frac{1}{2(q-1)} + \int_0^\infty \frac{\log^q(e^t)}{e^{2t}}e^t\dt  \\
	&= \frac{1}{2(q-1)} + \int_0^\infty t^qe^{-t}\dt\\
	&= \frac1{2(q-1)} + \Gamma(q+1)
\end{align*}
\end{proof}

\bibliographystyle{alpha}
\bibliography{../Lavrentiev/Lavrentiev_bibliography, 
	../../NN_bibliography}

\end{document}